\documentclass[11pt, a4paper]{amsart}
\usepackage{amsmath,amssymb,amsthm,mathtools,wasysym,calc,verbatim,tikz,url,hyperref,mathrsfs,cite,fullpage,bbm}
\usepackage[noabbrev,capitalize,nameinlink]{cleveref}
\usepackage[shortlabels]{enumitem}

\usepackage{comment}

\newtheorem{theorem}{Theorem}[section]
\newtheorem{lemma}[theorem]{Lemma}

\newtheorem{prop}[theorem]{Proposition}
\newtheorem{observation}[theorem]{Observation}

\newtheorem{fact}[theorem]{Fact}

\theoremstyle{definition}

\theoremstyle{remark}

\newtheorem*{remark*}{Remark}

\newcommand\R{\mathbb{R}}
\newcommand\Z{\mathbb{Z}}

\newcommand\cB{\mathcal{B}}

\newcommand\cQ{\mathcal{Q}}
\newcommand\cU{\mathcal{U}}

\newcommand\cE{\mathcal{E}}

\newcommand\eps{\varepsilon}

\renewcommand{\leq}{\leqslant}
\renewcommand{\geq}{\geqslant}
\renewcommand{\le}{\leqslant}
\renewcommand{\ge}{\geqslant}
\renewcommand{\to}{\rightarrow}
\renewcommand{\Re}{\re}

\def\eps{\varepsilon}

	\def\C{\mathbb{C}}
		
		\def\cF{\mathcal{F}}
		
	\def\R{\mathbb{R}}
		
	\def\Z{\mathbb{Z}}
	
	\def\PP{\mathbb{P}}
	\def\1{\mathbbm{1}}
	\def\l{\lambda}

	\def\s{\sigma}

	\def\g{\gamma}

	\def\la{\langle}
	\def\ra{\rangle}

	\def\Var{\mathrm{Var}}
	\def\Ber{\mathrm{Ber}}

	\def\tr{\mathrm{tr}}

  \def\supp{\mathrm{Supp\,}}

	\def\cB{\mathcal{B}}

	\def\EE{\mathbb{E}}

	\def\cQ{\mathcal{Q}}

\def\col{\mathrm{Col}}
\def\row{\mathrm{Row}}

\newcommand{\snorm}[1]{\lVert#1\rVert}

\renewcommand{\Re}{\operatorname{Re}}
\renewcommand{\Im}{\operatorname{Im}}

\newcommand{\mb}{\mathbb}

\newcommand{\mc}{\mathcal}

\newcommand{\mr}{\mathrm}

\newcommand{\on}{\operatorname}

\newcommand{\wt}{\widetilde}

\begin{document}

\title{The sparse circular law, revisited}

\author[A1]{Ashwin Sah}
\address{Massachusetts Institute of Technology. Department of Mathematics.}
\email{asah@mit.edu}

\author[A2]{Julian Sahasrabudhe}
\address{University of Cambridge. Department of Pure Mathematics and Mathematical Statistics.}
\email{jdrs2@cam.ac.uk}

\author[A3]{Mehtaab Sawhney}
\address{Massachusetts Institute of Technology. Department of Mathematics.}
\email{msawhney@mit.edu}
\thanks{Sah was supported by the PD Soros Fellowship. Sah and Sawhney were supported
by NSF Graduate Research Fellowship Program DGE-2141064. Part of this work was conducted when Sawhney was visiting Cambridge with support from the Churchill Scholarship.}

\begin{abstract}
Let $A_n$ be an $n\times n$ matrix with iid entries distributed as Bernoulli random variables with parameter $p = p_n$. Rudelson and Tikhomirov, in a beautiful and celebrated paper, show that the distribution of eigenvalues of $A_n \cdot (pn)^{-1/2}$ is approximately uniform on the unit disk as $n\rightarrow \infty$ as long as $pn \rightarrow \infty$, which is the natural necessary condition.  

In this paper we give a much simpler proof of this result, in its full generality, using a perspective we developed in our recent proof of the existence of the limiting spectral law when $pn$ is bounded. One feature of our proof is that it entirely avoids the use of $\eps$-nets and, instead, proceeds by studying the evolution of the singular values of the shifted matrices $A_n-zI$ as we incrementally expose the randomness in the matrix. 
\end{abstract}

\maketitle

\vspace{-1em}

\section{Introduction}

For an $n\times n$ matrix $M$, we define the \emph{spectral distribution} of $M$ to be the probability measure $\mu_M$ that puts a point mass of equal weight on each eigenvalue of $M$:
\[ \mu_M = n^{-1}\sum \delta_{\l}.\] The
study of the spectral distribution of \emph{random} matrices $M$ goes back 
to the seminal work of Wigner \cite{Wig58} in the 1950s who showed that the spectral distribution of random \emph{symmetric} matrices (so called \emph{Wigner} matrices) converges to the, so-called, semi-circular law in the large $n$ limit, after an appropriate rescaling. 

Determining the limiting spectral distribution for matrices with \emph{iid} entries
proved to be substantially more difficult and was only resolved by Tao and Vu \cite{TV10} after a long succession of important papers going back to the 1960s \cite{Meh67,Ede88,Gir84,Bai97,GT10,TV08}. They showed that the spectral measure for such matrices, after appropriate rescaling,  tends to the \emph{circular law} as $n\rightarrow \infty$, which is the probability measure that is uniform on the unit disc in $\C$.

\begin{theorem}[{Tao and Vu}] \label{thm:tao-vu-circular-law}
Let $\xi$ be a complex random variable with mean $0$ and variance $1$. For each $n$, let $A_n$ be a random matrix with iid entries distributed as $\xi$.
If we put $A^{\ast}_n = A_n \cdot n^{-1/2}$ then the spectral measure 
$\mu_{A_n^{\ast}}$ converges to the circular law in probability.
\end{theorem}

Actually Tao and Vu also proved Theorem~\ref{thm:tao-vu-circular-law} for the stronger notion of almost sure convergence but we express their theorem in terms of convergence \emph{in probability} as we will be exclusively interested in this form of convergence. 
Indeed, we say a  sequence of random measures $\mu_n$ converges to the circular law \emph{in probability} if for all $s,t\in \mb{R}$ and $\eps>0$, we have
\[\lim_{n\rightarrow \infty} \mb{P}\bigg[\bigg|\mu_n((-\infty,s)\times (-\infty,t)) - \frac{1}{\pi}\int_{-\infty}^s\int_{-\infty}^t\1_{x^2+y^2\le 1}~dydx\bigg|\ge \eps\bigg] = 0.\]

While Theorem~\ref{thm:tao-vu-circular-law} gives us a very good understanding of the limiting spectral law of the spectrum of ``dense'' matrices, it does not tell us anything about matrices where the non-zero entries are sparse. Of particular interest are \emph{iid Bernoulli matrices} where all entries are iid and distributed as\footnote{Here $\Ber(p)$ denotes, as is standard, a $\{0,1\}$-Bernoulli random variable taking $1$ with probability $p$.} $\Ber(p)$ for $p = p_n \rightarrow 0$.

This sparse setting was considered by G\"{o}tze and Tikhomirov \cite{GT10}, who proved that the limiting spectral distribution is still the circular law when $p > n^{-1/4+\eps}$. Tao and Vu \cite{TV08} improved this range to cover all $p > n^{-1+\eps}$, and Basak and Rudelson further improved this to cover all $p > \omega(n^{-1}(\log n)^2)$.

Then, in a difficult and celebrated paper, Rudelson and Tikhomirov \cite{RT19} proved that $pn \rightarrow \infty$ is sufficient for convergence to the circular law, which is also the natural necessary condition: if $pn$ is bounded it is easy to see that the limiting measure must have a large atom at zero. 
 
In fact, they prove a more general result that allows for the non-zero entries to be replaced with iid copies of a random variable $\xi$ of variance $1$. Indeed let $A_n \sim \Delta_{n}(p,\xi)$ indicate that $A_n$ is a $n\times n$ random matrix where each entry is an iid copy of $\Ber(p) \cdot \xi$.

\begin{theorem}[Rudelson and Tikhomirov] \label{thm:main}
Let $\xi$ be a real random variable with $\mb{E}\, \xi^2=1$, let $np \rightarrow \infty$ and $p\to 0$ and for each $n$, let 
$A_n\sim\Delta_{n}(\xi,p)$. If we put $A^{\ast}_n = A_n \cdot (pn)^{-1/2}$ then $\mu_{A^{\ast}_n}$ converges to the circular law in probability.
 \end{theorem}

This long line of results leaves open the case of $p = d/n$, for constant $d >0$, which has proven to be the most difficult and subtle case. In our paper \cite{SSS23b}, we complete this program by proving the existence of the limiting measure in this case.

\begin{theorem}\label{thm:sparse-law}
For $d > 0$ and each $n$, let $A_n$ be an $n\times n$ matrix with iid entries distributed as $\on{Ber}(d/n)$. There exists a distribution $\mu_d$ on $\C$ so that $\mu_{A_n}$ converges to $\mu_d$, in probability. 
\end{theorem}

In this paper we give a new and considerably shorter proof of Theorem~\ref{thm:main} based on the method which we used in \cite{SSS23b} to prove Theorem~\ref{thm:sparse-law}.

The method and ``philosophy'' of our proof is considerably different from that of Rudelson and Tikhomirov. We 
completely avoid any direct use of $\eps$-nets and, instead, favour of a more ``dynamic'' approach, where we track the evolution of the point processes defined by the singular values of the shifted matrices $A-zI$ as
we expose a new rows and columns. To pull off this analysis we need only to rely on a few ``quasi-randomness'' conditions on the graph defined by the non-zero entries.

One added advantage of our approach is that it allows us to effortlessly generalize Theorem~\ref{thm:main} of Rudelson and Tikhomirov to allow for complex-valued random variables $\xi$ with unit variance. It appears the work of Rudelson and Tikhomirov would not generalize to this complex case without significant new ideas. Thus our main theorem here is the following.

\begin{theorem} \label{thm:main-complex}
Let $\xi$ be a complex random variable with $\mb{E}\, |\xi|^2=1$, 
$np \rightarrow \infty$ and $p\to 0$ and for each $n$, let $A=A_n\sim\Delta_{n}(\xi,p)$. If we put $A^{\ast}_n = A \cdot (pn)^{-1/2}$ then $\mu_{A^{\ast}_n}$ converges to the circular law, in probability.
 \end{theorem}

We now turn to sketch the proof of Theorem~\ref{thm:main-complex} and set up the remainder of the paper. 

\section{Description of method}\label{sec:sketch}

To establish the convergence of the spectral law to the circular law it is enough to prove the convergence of
the logarithmic potential of the spectral law to the logarithmic potential of the circular law. For this, we may use Girko's ``hermitization'' method (see e.g. \cite{BC12}) to express the logarithmic potential of the spectral law as the (random) function
\begin{equation} \label{eq:girko}
U_{n}(z) =  -\frac{1}{n} \sum_{j = 1}^{n} \log\big(\sigma_j\big(A_{z}^{\ast}\big)\big),
\end{equation} 
where we set $d = pn$, here and throughout the paper, and define  
\[ A_{z}^\ast = A^\ast-zI_n = d^{-1/2} A-zI_n.\]
Here we have also used the notation 
$\sigma_1(M) \geq \sigma_2(M) \geq \cdots \geq \s_m(M)$, to denote the (right) singular values of the $n \times m$ matrix $M$. Our main task is to prove that for all $z \in \C\setminus \{0\}$ we have 
\begin{equation}\label{eq:point-wise}
\lim_{n}\, U_n(z) =  U^{\circ}(z) = \begin{cases} \, 
-\log|z|,&\emph{if }|z|\ge 1 ;\\
\, (1-|z|^2)/2,&\emph{if }|z|\le 1, \end{cases}
\end{equation}
with probability $1$, where $U^{\circ}$ is the log potential of the circular law. Once we have proved this, we can simply appeal to the general theory of logarithmic potentials to conclude the convergence in probability $\mu_n \rightarrow \mu^{\circ}$.
  
Given the expression \eqref{eq:girko}, there are traditionally two steps to establish \eqref{eq:point-wise}. First, one shows that the limit of the empirical distribution of the \emph{singular values} $\sigma_j(A_{z})$ is ``what it should be'' by a (fairly standard) method of moments and truncation argument: this tells us that the ``bulk'' of the sum in \eqref{eq:girko} converges to what it is supposed to. Second, one shows that the small singular values 
$\s_n,\s_{n-1},\ldots$ don't spoil the bulk convergence of this sum by getting too small and dominating the sum~\eqref{eq:girko}. Since we have good tools for establishing the first step these days, it is this second step that represents the core challenge and, in particular, one essentially needs to prove bounds of the type 
\begin{equation}\label{eq:sing-val-bound}
\PP\big( \sigma_{n-k}(A_{z}^\ast) \leq \exp(-\eps n/(k+1) ) \big) = o(1),
\end{equation}
for any $\eps >0$ and all $k =0,1,2,\ldots$.
  
Now, \emph{heuristically} we expect that typically $\sigma_{n-k} = \Theta( (k+1) d^{1/2} n^{-1})$, and thus \eqref{eq:sing-val-bound} may not appear to be a particularly difficult obstacle to overcome, as it represents what we expect to be an extremely abnormal behaviour. However, obtaining bounds of this type has recently represented \emph{the} significant challenge in this area. Indeed, bounds of this type represent one of the main achievements of the work Tao and Vu \cite{TV10} in their work on the circular law for dense matrices. For sparse matrices, the challenge is greater still as there is less ``randomness'' to use. For their sparse circular law, Rudelson and Tikhomirov \cite{RT19}, develop a whole toolbox of sophisticated techniques for constructing $\eps$-nets to prove singular value estimates of the type\footnote{In fact, both Tao and Vu \cite{TV10}  well as Rudelson and Tikhomirov \cite{RT19} obtain better estimates on the least singular value than required, but this is not relevant to our discussion here.} \eqref{eq:sing-val-bound}. 
 
In this paper, we take a more direct route to proving \eqref{eq:point-wise} that is substantially different and considerably simpler. Instead of directly working with the singular values, we look to \emph{compare}  $U_n$ with a ``truncated'' version of the log potential of a principal minor of $A^\ast$. To elaborate on this, let $\eps \rightarrow 0$ sufficiently slowly and set $m = (1-\eps)n$. We will understand $A_{m,z}^\ast$ to be the top left $m \times m$ principal minor of $A_{z}^\ast$. The main objective of the proof will be to compare $U_n(z)$ with the \emph{truncated} sum 
\[ T_n(z) = -\frac{1}{n}\sum^{(1-\eps/4)m}_{j=1}\log\big(\sigma_j\big(A_{m,z}^\ast\big)\big).\]
The point here is that the sum $T_n(z)$ is much easier to deal with than the corresponding log potential: the smallest $\eps m/4$ singular values have been removed from the sum. Thus a simple variant of the trace moment method is sufficient to establish
\[T_n(z)=U^{\circ}(z)+o(1)\]
with high probability as $n\to\infty$. The core of the proof, therefore, lies in making the comparison between $U_n$ and $T_n$, and in particular showing $U_n(z) \leq T_n(z)+o(1)$, which we achieve dynamically: we ``build up'' $A_{n,z}^\ast$ from $A_{m,z}^\ast$ by alternately adding rows and columns. Simultaneously,
we build up $U_n$ from $T_n$ by taking more singular values into our sum, when possible. For this, we index time in \emph{half-integer} steps $t \in [m,n]$ so that at each integer time $t$ we have that $A_t^\ast$ is a $t\times t$ matrix and at time $t+1/2$, we define $A_{t+1/2}^\ast$ to be the $t \times (t+1)$ matrix which is $A_{t+1}^\ast$ with the bottom-most row deleted. Thus we build $A_{n,z}^\ast$ from $A_{m,z}^\ast$ as
\[ A_{m,z}^\ast \rightarrow A_{m+1/2,z}^\ast \rightarrow \cdots \rightarrow A_{n-1/2,z}^\ast \rightarrow A_{n,z}^\ast. \]
The point of this is that each row and column addition has the crucial property that it ``pushes'' all the singular values up and thus we aim to take more singular values into our sum as these values get ``pushed'' throughout the process. Let us define the sum 
\[ T_{r,t}(z) = -\frac{1}{n} \sum_{j = 1}^{r} \log\big(\sigma_j\big(A_{t,z}^\ast\big)\big). \]
Our crucial lemma tells us that we can take on singular values rather often: if $A_{t,0}$ satisfies some quasi-randomness conditions and $r < t$ then 
\begin{equation}\label{eq:step-in-walk} 
\PP\big( T_{r+1,t+1/2}( z ) \leq T_{r,t}(z) + \delta_{r+1} \big) = 1-o_{d\rightarrow \infty}(1),
\end{equation}
where the probability is only over the new row/column being added and $\delta_r$ is a sequence with the property that 
\begin{equation} \label{eq:delta-sum}
\sum_{r=(1-\eps/4)m}^n \delta_r = o_{\eps \rightarrow 0}(1).
\end{equation}
However this is not the end of the story; every time we add a column to our matrix we ``create'' a new singular value at $0$ that needs to be taken under control in the following steps. Thus we need to ensure that our random process has sufficient drift to ensure that we can take on all new singular values by the end of the process. 

To see this is the case, consider the random process $r(t)$, which is defined as the number of singular values that we have taken on by time $t$. It makes sense to consider the ``height'' of this process defined by $h(t) = t - r(t)$, which represents the number of singular values we don't have in our possesion at time $t$. Thus the goal of our work can be phrased as showing that $h(n) =0$ with high probability. Note that at the start of this process, we have 
\[ h(m) = m - (1-\eps/4)m = \eps m/4 \]
and at each step we have a downward drift of $(1/2-o(1))$ with each row or column addition. Since there are $2\eps n$ row or column additions we expect the total drift to be $\eps n$ which is significantly larger than $\eps n/4$, and should give $h(n) = 0$ whp. Of course the wrinkle is that we need to ensure that these quasi-randomness conditions, mentioned above, occur sufficiently often.

To prove \eqref{eq:step-in-walk} we are led to study the structure of the vector space spanned by the $\lceil t\rceil-r$ smallest singular vectors of $M$ (left-singular vectors for integral $t$ and right-singular vectors for half-integral $t$). In particular, we will need to show that if $X$ is a new row or column of our matrix and $P_{h,M}$ is the orthogonal projection onto the space spanned by the $h$ smallest singular vectors of the appropriate side then for ``quasi-random'' $M$ we have\footnote{Here are throughout this paper, we shall write $\log_{(2)} x = \log\log x$.}
\begin{equation}\label{eq:ker-anti-concentration}
\PP_X\big( \|P_{\lceil t\rceil-r,M} X\|_2 < \exp(-n\delta_r) \big) = O(\eps +  (\log_{(2)} d )^{-1/2}).
\end{equation}
While theorems of this type can be extremely challenging in general, here we can get away with a very weak notion of quasi-randomness, based on the number of rows that have a unique non-zero entry in a set of columns. From this quasi-randomness condition, we can deduce the following basic structural information about vectors $v \in \C^t$ that are near-singular vectors of $M$:
\begin{equation}
\label{eq:ker-struct} 
|\{ i\colon v_i \ge \exp(-n\delta_r) n^{-1/2} \}| \geq (n/(2d))\log_{(2)} d.
\end{equation}
This in, in turn, allows us to deduce \eqref{eq:ker-anti-concentration}.

In Section~\ref{sec:UNE} we define the notion of a ``unique neighbourhood expansion'' which our quasi-randomness notion is based on. In Section~\ref{sec:ker-vectors} we use this quasi-randomness property to derive~\eqref{eq:ker-struct}. In Section~\ref{sec:proj-anti-concentration} we prove that if the kernel has the property \eqref{eq:ker-struct} then we can deduce \eqref{eq:ker-anti-concentration}. Then in Section~\ref{sec:a-step} we prove \eqref{eq:step-in-walk}, showing how the process evolves in a single step. In Section~\ref{sec:analysis} we give the simple analysis of this random walk. In Section~\ref{sec:completion} we show the convergence of $T_n$ to $U^{\circ}$ and complete the proof of Theorem~\ref{thm:main-complex}.

\section{A few preliminaries}

In what follows we fix $\xi$ to be a complex random variable with
\[ \EE\,|\xi|^2  = 1 .\]
We regard $\xi$ as fixed throughout the paper and allow various quantities to depend on $\xi$. Given such a random variable $\xi$, we define $\beta = \beta(\xi)\leq 1$ to be such that 
\[ \max_{y}\, \PP\big( \|\xi - y\| < \beta \big) \leq 1-\beta.\]
Note that either $\beta(\xi)>0$  or $\beta(\Ber(1/2)\cdot \xi)>0$ and thus adjusting $p$ and rescaling if necessary we may assume that $\beta(\xi)>0$.

We define $\Delta_{n,m}(p,\xi)$ to be the probability space of all $n\times m$ matrices of $A$ where $A_{i,j} = \delta_{i,j}\xi_{i,j}$, all of the $\xi_{i,j}$ and $\delta_{i,j}$ are independent, and 
$\delta_{i,j} \sim \Ber(p)$ and $\xi_{i,j}$ is distributed as $\xi$.
We define $\Delta_{n}(\xi,p) = \Delta_{n,n}(\xi,p)$. We define $\col_n(\xi,p)$ to be the random column vector $X \in \R^n$ where $X_i$ are independent and distributed as $\Ber(p) \cdot \xi$. We define $\row_n(\xi,p)$ similarly. For a matrix $M$ we also define, as is standard, the \emph{Hilbert-Schmidt norm} to be $\snorm{M}_{\mr{HS}}^2 = \sum_{i,j}|M_{i,j}|^2$.

We define for easy reference the quantities $\delta_r$ mentioned in the 
proof outline. We define 
\begin{equation}\label{eq:def-delta_r} \delta_r = \begin{cases} \hspace{2mm} n^{-1} (\log(n/(n-r+1)))^{2}&\qquad  \text{for }  r < n(1-d^{-1/4})  ;\\
                 \hspace{2mm}     C n^{-1}(\log d)^8(\log(n/(n-r+1)) )^8 &\qquad  \text{for }  r \geq n(1-d^{-1/4}) .  \end{cases} \end{equation}
As we required in \eqref{eq:delta-sum} it is not hard to check that $\sum_{r=(1-\eps/4)n}^n \delta_r = O(\eps)$.  To save on clutter, it also makes sense to define 
\[\eta_r = \exp(-n\delta_r).\] 

In addition to the fixing of $\xi$, we make a few more global assumptions throughout the paper. We assume throughout $n$ is sufficiently large and that $Cn^{-1} \leq p \leq 1/2$, where $C$ is a sufficiently large constant depending only on $\xi$. We assume $t \in \frac{1}{2}\mathbb{Z} \cap [m,n]$ and $m = (1-\eps)n$. Throughout, for $t\in\mb{N}$, $A_t$ will be our $t\times t$ random matrix for time $t$ and $A_{t+1/2}$ will be a $t\times(t+1)$ random matrix, obtained as submatrices of $A$. Throughout we will define $A_{t,z} = A_t - zI$, where $I$ is either a square identity or $(t-1/2) \times (t+1/2)$ identity matrix depending if we are at an integer or half-integer time. As in the previous section, define $A_{t,z}^{\ast} = d^{-1/2}A_t - zI$. We also allow ourselves the convention that $d = pn$ and that $z \in \C$. All whp statements are meant with respect to $n\to\infty$ and therefore $d\to\infty$.

\section{Unique neighbourhood expansions and quasi-randomness properties}\label{sec:UNE}

The  goal of this section is to define the quasi-randomness event $\cE_r$ that will allow us to show that we can take on the $r$th singular value into our sum, with high probability, assuming $r \leq \lfloor t\rfloor$. We define this event $\cE_{r}$ in three parts
 \[ \cE_{r}(\xi,p) = \cU_{r}(\xi,p) \cap \cB(p) \cap \cQ(\xi,p) \cap \mc{R}(\xi,p).  \] Even before defining $\cE_{r}$ we state the main lemma of this section which is essentially the only thing we need to carry forward in the paper.

\begin{lemma}\label{lem:quasi-random}
If $r \geq t - n/d^{1/4}$ then 
\[ \PP(A_t \in \cE_{r} ) \geq 1-\exp(-d^{1/2}(\lceil t\rceil-r+1))-n^{-3} .\]\end{lemma}

We now define $\cU_{r}$. For this, we define the ``unique neighbourhood'' of a set of columns. Here we are borrowing  terminology from the graph theoretic interpretation of our matrix where rows and columns represent vertices in a bipartite graph. Let $B$ be an 
$m \times \ell$ matrix and let $S \subseteq [\ell]$ be a set of indexes of columns of $B$. We define the \emph{unique neighbourhood} of $S$, which we denote $U(S)\subseteq[m]$ (a subset of the rows of $B$), in two parts. We first define 
\[ U(S) \setminus S 
= \big\{ i \in [m]\setminus S\colon B_{i,j}\neq 0 \text{ for a unique } j \in S \text{ and } 
|B_{i,j}|\geq \beta    \big\} \]
and then define
\[  U(S) \cap S = \big\{i\in[m]\cap S\colon B_{i,j} = 0 \text{ for all } j \in S \big\} . \] 

The motivation behind this definition comes from the observation
that if a vector $v\in \mb{R}^{\ell}$ is non-zero exactly on $S$ then $Bv$ must be non-zero on $U(S)\setminus S$.  A quantitative analog of this lemma appears as Observation~\ref{obs:NU-obs}.  

The event $\cU_r$ is the event that $U(S)$ is large for all moderately sized sets of columns $S$. In particular, fix $\alpha(x) = ( \log (n/x) )^{-2}$ and say $B \in \cU_{r}$ if for all subsets of columns $S$ with 
\[ c^\ast(\lceil t\rceil-r+1) \leq |S|  \leq (n/(2d))\log_{(2)} d \quad \text{ we have }  \quad |U(S)| \geq  \alpha(|S|) d|S|,\]
where $c^\ast>0$ is the absolute constant appearing in the statement of Lemma~\ref{lem:basis}, the precise value of which is not particularly important. 

The observation behind this definition is that we can use it to get some basic control over vectors that are in the kernel of our matrix. Indeed if 
$B \in \cU_r$ then 
\[ Bv = 0 \qquad \Longrightarrow \qquad  |\supp(v)| >  (n/(2d))\log_{(2)} d  \quad \text{ or } \quad |\supp(v)| < c^\ast(\lceil t\rceil-r+1). \] 
Actually we will need a more complicated version of this observation, which works for vectors that are near to kernel vectors.  

The following lemma tells us that $\cU_r$ holds with sufficiently high probability.

\begin{lemma}\label{lem:unstructured-graph}
If $r \geq t - n/d^{1/4}$ then
\[\PP( A_t \in \cU_{r} ) \geq 1-\exp(-2d^{1/2}(\lceil t\rceil-r+1)).\]
\end{lemma}

Now define $A_t \in \cB$ if for all subsets $S\subseteq[\lfloor t\rfloor]$ of the rows, we have 
\[ \frac{1}{|S|}\sum_{i\in S}\sum_{j=1}^{\lceil t\rceil}\delta_{i,j} = O(d + \log(n/|S|)), \]
and analogously for the columns.

Furthermore, define $A_t \in \cQ$ if there are at most 
\[  2dn/H^2 + (\log n)^2  \text{ entries of } A_t \text{ which are }  > 8H/\beta \text{ in magnitude}\] for each $H\in[1,n^4]$, and additionally the maximum value of $A_t$ is at most $n^3$.

Finally define $A_t \in \mc{R}$ if for all $\ell\ge 1$, if we put $L = ( d n/(\beta \ell))^5$ then
\[ \big| \big\lbrace i\colon \sum_{j=1}^{\lceil t\rceil} |(A_{t})_{i,j}| > L/\beta \big\rbrace \big| \le \alpha(\ell) d\ell/4\]
and similarly for columns.

\begin{lemma}\label{lem:neighbor-size}
We have
\[ \PP( A_t \in \cB \cap \cQ \cap \mc{R}) \geq 1-n^{-3}.\]
\end{lemma}
Note by construction that it suffices to check Lemma~\ref{lem:neighbor-size} only when $t=n$. The proof of all of these lemmas are entirely standard union bound computations and are deferred to Appendix~\ref{sec:facts}.

\section{Spreadness of near-kernel vectors from graph quasi-randomness}\label{sec:ker-vectors}

The goal of this section is to prove the following lemma, which says that if $A_t \in \cE_r$ and $v$ is ``close'' to the small right-singular vectors of $A_{t,z}$ then it has $(n/(2d))\log_{(2)} d \gg n/d$ coordinates that are bounded away from zero by $\eta_r n^{-1/2}$. The point here is that each new row and column has random support of size $d$, on average. So, crucially, the intersection of these supports is typically $\gg 1$. For this lemma we define the notation, for $v \in \C^n$ and $x \geq 0$, $\l(v; x) = |\{ i\colon |v_i| \geq x \}|$.

\begin{lemma}\label{cor:spread-vector}
For $t\geq r \geq n(1-d^{-1/4}) $, let $A_t \in \cE_{r}$ and $1\leq |z| \leq d$. Put $k = c^\ast(\lceil t\rceil-r+1)$, where $c^\ast$ is as in Section~\ref{sec:UNE}, and let $v \in \C^{\lceil t\rceil}$ satisfy
\[ \snorm{A_{t,z} v}_2\le d^{1/2}\eta_{r} \hspace{1mm} \emph{ and } \hspace{2mm} \l(v ; k^2n^{-5/2}) \geq k     .\]
Then
\[ \lambda(v ; \eta_{r}d^{1/2}k^{-1/2}) \geq (n/(2d))\log_{(2)} d.\]
\end{lemma}

\vspace{2mm}

\noindent

We prove this lemma by using the unique neighbourhood expansions. This link starts to become apparent with the following simple observation. 

For this we define a few uses of notation. For $v \in \C^n$ and $S \subseteq [n]$, we let $v_S \in \C^n$ be the vector with $(v_S)_i= v_i$ for $i \in S$ and $0$ otherwise. It is also 
useful to define, for a vector $v = (v_1,\ldots,v_n)$, the vector $v^{\ast} = (v_{1}^{\ast},\ldots, v_n^{\ast})$ where the entries of $v^{\ast}$ are the $|v_i|$, but have been permuted so that $v^{\ast}_1 \geq \cdots \geq v^{\ast}_n$.

\begin{observation}\label{obs:NU-obs}
For $\ell\leq t$, let $v \in \C^{\lceil t\rceil}$ and let $S$ be the set of the $\ell$ largest coordinates of $v$ in absolute value. If $|z| \geq 1$ then $|(A_{t,z}v_S)_i| \geq v^{\ast}_{\ell}\beta$ for all $i \in U(S)$. 
\end{observation}
\begin{proof}
We consider two cases. If $i\in U(S)\setminus S$ there is unique $j\in S$ with $(A_{t,z})_{i,j}\neq 0$. For this $j$, we additionally have $|\xi_{i,j}|\ge \beta$. Thus the observation is easily proved in this case.

For $i\in U(S)\cap S$ we have $(A_{t,z})_{i,j} = 0$ for all $j\in S$ with $j \not= i$. So 
\[|(A_{t,z}v_S)_i|=|(A_{t,z})_{i,i}v_i| = |(A-zI)_{i,i}v_i|=|z||v_i| \geq v_\ell^\ast\beta,\]
which proves the observation.
\end{proof}

We remark briefly that in the above lemma we restricted attention to $|z|\ge 1$. When considering convergence to the circular law, we will need to normalize $A_{t,z}$ by a factor of $d^{-1/2}$ and thus the condition $|z|\ge 1$ will effectively become $|z|\ge d^{-1/2}$. This is crucial as we may only exclude a measure $0$ set of $z$ when proving convergence to the circular law. (Furthermore when working with $A_{t,z}^{\ast}$ therefore we will only exclude $|z|\ge d^{-1/2}$; this difference corresponds precisely to this rescaling.)
\vspace{2mm}

Recall that above we defined the function $\alpha(x) = (\log (n/x) )^{-2}$. Here we define the function 
\begin{equation}\label{eq:g(k)-def}  
g(x)= \left\lceil \frac{d\alpha(x)x}{C'(d+\log (n/x))}\right\rceil,
\end{equation} 
where $C'$ is a sufficiently large constant. (It is chosen based on the $O(\cdot)$ in the definition of the event $\cB$.)

We prove Lemma~\ref{cor:spread-vector} by iterating the following lemma. 
This lemma says that if the mass of $v$ is clustered on fewer than $(n/(2d))\log_{(2)}d$ coordinates then $A_{t,z}v$ has many large coordinates and is therefore is not close to the small singular vectors.

\begin{lemma}\label{lem:decreasing-coordinate}
For $r \leq t$ let $A_t \in \cE_r$ and let $c^\ast(\lceil t\rceil-r+1) \leq \ell \leq (n/(2d))\log_{(2)}d$, where $c^\ast$ is as in Section~\ref{sec:UNE}. If $v \in \C^{\lceil t\rceil} $ satisfies  
\[ v^{\ast}_{\ell+g(\ell)} \leq v^{\ast}_{\ell} \left( \frac{\beta \ell}{ d n} \right)^7.\] 
and $1\leq |z| \leq d$ then
\begin{equation} \label{eq:large-coords}
\l\big( A_{t,z}v\, ;\, \beta v_{\ell}^{\ast}/2 \big) \ge \alpha(\ell)d\ell/2. \end{equation}
\end{lemma}
\begin{proof}
Let $S$ be the set of the $\ell$ largest coordinates of $v$ in absolute value. Write $v = x + y$, where $x = v_S$ and $y = v_{S^c}$ and note that by  Observation~\ref{obs:NU-obs} we have 
\[ |( A_{t,z}x)_i| \geq |(A_{t,z}x)_i| - |(A_{t,z}y)_i| \geq v^{\ast}_\ell\beta - |(A_{t,z}y)_i|, \]for all $i \in U(S)$. Since $A_t \in \cE_r$ we have that $|U(S)| \geq  \alpha(\ell)d\ell$ and thus it suffices to prove
\begin{equation}\label{eq:few-bad-in-NU}
\big|\big\{ i \in U(S) \colon |(A_{t,z}y)_i| > \beta v_\ell^{\ast}/2 \big\}\big| < \alpha(\ell)d\ell/2.
\end{equation} 
To prove \eqref{eq:few-bad-in-NU}, let $B$ be the set in \eqref{eq:few-bad-in-NU} and define $S^\ast \supseteq S$ to be the set of indices of the $\ell+g(\ell)-1$ largest coordinates of $v$ in magnitude.  If $i \in B$ then either row $i$ has unusually large magnitude \emph{or} row $i$ has a non-zero entry in $S^{\ast}\setminus S$. More precisely, define 
\[B_1  = \bigg\lbrace i \colon \sum_{j=1}^{\lceil t\rceil}|\delta_{i,j}\xi_{i,j}|\ge L/\beta \bigg\rbrace \hspace{2mm} \text{ and } \hspace{2mm} B_2 = \bigg\lbrace i  \colon \sum_{j\in S^\ast\setminus S}\delta_{i,j}>0 \bigg\rbrace,\] 
where we have set $L = ( d n/(\beta \ell))^5$. We now claim that $ B \subseteq B_1 \cup B_2 $. To see this, assume $i \notin B_1 \cup B_2$ and observe that 
\[|(A_{t,z}y)_i| \le\sum_{j\notin S}|\delta_{i,j}\xi_{i,j}-z\1_{j=i}||v_j|=\sum_{j\notin S^\ast}|\delta_{i,j}\xi_{i,j}-z\1_{j=i}||v_j|,\]
since $i \not\in B_2$. Using that $|v_{j}| \leq v_\ell^{\ast}L^{-1}(\beta \ell/(dn))^2$ for $j \not\in S^{\ast}$, we see the above is at most 
\[\left(\frac{\beta\ell}{dn}\right)^2 \cdot \frac{v_\ell^\ast}{L} \cdot \bigg(d+\sum_{j\notin S^\ast}|\delta_{i,j}\xi_{i,j}|\bigg)
\le\left(\frac{\beta}{d}\right)^2 \cdot \frac{v_\ell^\ast}{L} \cdot (d+ \beta^{-1}L)\le \beta v_\ell^\ast/2,\]
where we have used that $|z| \leq d $, $\ell\leq n$, and $i \notin B_1$. Thus $B \subseteq B_1 \cup B_2$.

To conclude \eqref{eq:few-bad-in-NU}, we just need to show $|B_1|,|B_2| < \alpha(\ell) d \ell/4$. Since $A_t \in \cE_r$, the definition of event $\cB$ tells us
\[|B_2| \leq \sum_{i \in S^{\ast}\setminus S} \sum_{j} \delta_{i,j} 
\leq O(g(\ell)(d + \log(n/g(\ell)))) \leq \alpha(\ell)d \ell /4.\]
Recall that in the definition of $g$, $C'$ is chosen sufficiently large. Additionally, the definition of event $\mc{R}$ tells us $|B_1| \leq \alpha(\ell)d\ell/4$. This completes the proof.
\end{proof}

\vspace{2mm}

We now iterate Lemma~\ref{lem:decreasing-coordinate} to obtain Lemma~\ref{cor:spread-vector}. To understand how many times we need to iterate Lemma~\ref{lem:decreasing-coordinate}, we need the following basic numerical fact. For this, we think of $k$ as in Lemma~\ref{cor:spread-vector} and we define the sequence $(k_t)_{t\geq 0}$, by setting $k_0 = k$ and then defining
\[ k_i = k_{i-1} + g(k_{i-1}),\]
for all $i \geq 1$.

\begin{fact}\label{fact:number-of-iterations}
Let $\tau$ be the minimum value for which $k_{\tau} > (n/(2d))\log_{(2)} d$. Then $\tau = O( (\log (n/k) )^4 )$.
\end{fact}

We prove this fact in Appendix~\ref{sec:facts} and now jump to the proof of Lemma~\ref{cor:spread-vector}.

\begin{proof}[Proof of Lemma~\ref{cor:spread-vector}]
We let $k_t$ and $\tau$ be as above. We claim that for all $i\in[\tau]$ we have
\begin{equation}\label{eq:vector-decay}
v_{k_i}^\ast\ge v_{k_{i-1}}^\ast \cdot \delta,
\end{equation}
where $\delta = (\beta k/(dn))^7$. For a contradiction assume $i\in[\tau]$ is the smallest failure of this inequality. We will then apply Lemma~\ref{lem:decreasing-coordinate} to show that this contradicts the assumption that $v$ is close to the small singular vectors. 

Indeed, the failure of \eqref{eq:vector-decay} allows us to apply Lemma~\ref{lem:decreasing-coordinate} to $v$ to learn
\begin{equation}\label{eq:lemma-app} 
\l\big( A_{t,z} v \,; \, \beta v^{\ast}_{k_{i-1}}/2 \big)\geq \alpha(k_{i-1})dk_{i-1}/2 \ge k (\log(n/k))^{-2}.
\end{equation}
Since $i$ is the minimum such value for which \eqref{eq:vector-decay} fails, we have 
\begin{equation}\label{eq:v_k-lb}
v_{k_{i-1}}^\ast\ge \delta^{(i-1)}v_{k_0}^\ast\ge \delta^{(i-1)}k^2 n^{-5/2},
\end{equation}
where the last inequality holds
by the given lower bound on $v_{k_0}^\ast=v_k^\ast$. Now \eqref{eq:lemma-app} and \eqref{eq:v_k-lb} taken together imply
\[\snorm{A_{t,z} v}_2^2 \ge k (\log(n/k))^{-2} \cdot \beta^2\delta^{2(i-1)}k^4 n^{-5}/4 \ge \delta^{2i}.\]
Since $i\le\tau\le O((\log(n/k))^4)$, we see that this contradicts the assumption 
\[ \|A_{t,z}v \|_2^2 \leq \eta_{r}^2d \leq \exp(-C(\log (n/k))^8(\log d)^8 ), \]
where we have used that $k = c^\ast(\lceil t\rceil-r+1)$ and $k \leq n/d^{1/4}$. So in fact \eqref{eq:vector-decay} holds for all $i\in[\tau]$, as claimed.

Now iterating \eqref{eq:vector-decay} we obtain the desired result, using that we have $\tau=O((\log(n/k))^4)$ and $k_\tau\ge (n/(2d))\log_{(2)} d$.
\end{proof}

\section{Projection anti-concentration}\label{sec:proj-anti-concentration}

As discussed in Section~\ref{sec:sketch}, our ability to ``take on'' the singular value $\s_{r}$ depends on the magnitude of the projection of our new row or column onto the space spanned by the vectors $u_n,\ldots,u_{r}$, corresponding to the smallest singular directions $\s_n,\ldots, \s_{r}$ (on the appropriate side). Here we prove that it is unlikely for this projection to be small assuming $A_t \in \cE_r$ (or $A_t^\dagger\in\cE_r$). Recall that for a $t \times t$ or $(t-1/2)\times(t+1/2)$ matrix $M$, we let $P_{r,M}$ be the orthogonal projection onto the $\lceil t\rceil-r+1$ smallest right-singular directions of $M$.

\begin{lemma}\label{lem:proj-anti-concentration} 
Let $t-2\eps n\leq r \leq\lceil t\rceil$, let $A_t$ be a $t \times t$ matrix, let $1 \leq |z| \leq d$, and put $M = A_{t,z}$. Moreover, if $\lceil t\rceil -r +1\le nd^{-1/4}$ additionally assume $A_t \in \cE_r$. If $X \sim \row_{\lceil t \rceil}(\xi,p)$ and $\sigma_{r}(M)\le d^{1/2}\eta_r$ then for all $w\in \mb{C}^{\lceil t \rceil}$,
\[ \PP_X\big( \|P_{r,M}(X+w)^\dagger\|_2 < d^{1/2}\eta_r \big) = O(\eps+(\log_{(2)} d )^{-1/2}). \]
\end{lemma}

The remainder of this section is devoted to a proof of this lemma. The proof is broken into different regimes, when the co-dimension is large, meaning $h = \lceil t\rceil-r+1 > nd^{-1/4}$, and when it is small, meaning $h \leq nd^{-1/4}$. We warm up by proving the large $h$ case, since this is significantly easier and does not need $A_t \in \cE_r$. When $h$ is small we will appeal to the results proved in the previous section about matrices $A_t \in \cE_r$.

\subsection{Proof of the large \texorpdfstring{$h$}{h} case}
We take care of the case of large $h$ with the following.

\begin{lemma}\label{lem:hlarge-proj-anti-concentration}
For $h \geq nd^{-1/4} $, let $X \sim \row_{\lceil t \rceil}(\xi,p)$ and let $P$ be an orthogonal projection onto an $h$-dimensional subspace. Then for all $\kappa>0$ and $w\in \mb{C}^n$ we have
\[ \PP_X\big( \|P (X+w)^{\dagger}\|_2 < \kappa \cdot d^{1/2}h^{3/2}n^{-3/2} \big) = O(\kappa + d^{-1/4}). \]
\end{lemma}

The following lemma of Litvak, Lytova, Tikhomirov, Tomczak-Jaegermann, and Youssef \cite[Lemma~4.3]{LLTTY21} gives a decent basis for this space. 

\begin{lemma}\label{lem:basis}
Let $V \subseteq \C^n$ be a $k$-dimensional $\C$-vector space. There exists an orthonormal basis $B$ of $V$ so that for all $v \in B$, we have $v^{\ast}_{c^\ast k} \geq  c^\ast k^{1/2}n^{-1}$, where $c^\ast>0$ is an absolute constant. 
\end{lemma}

We also need the following anti-concentration inequality, which is
a straightforward consequence of the classical inequality due to 
L\'evy--Kolmogorov--Rogozin \cite{Kol58,Rog61}. 

\begin{lemma}\label{lem:anti-concentration-bernoulli}
Let $X \sim \col_n(\xi,p)$, and let $v \in \C^{n}$ satisfy $v_k^{\ast} \geq \rho$.  Then 
\[ \max_{y \in \C}\, \PP\big( | \la X, v \ra - y | \leq r\big)\le\frac{C_{\beta}}{(kp)^{1/2}} \cdot \frac{r}{\rho},\]
for all $r \geq \beta \rho/\sqrt{2}$. Here we can take $C_{\beta} = O(\beta^{-3/2})$. 
\end{lemma}

We include the simple deduction of this lemma in Appendix~\ref{sec:facts}. We are now prepared to prove Lemma~\ref{lem:hlarge-proj-anti-concentration}, which takes care of Lemma~\ref{lem:proj-anti-concentration} for large $h$.

\begin{proof}[Proof of Lemma~\ref{lem:hlarge-proj-anti-concentration}]
Let $V$ be the image of $P$. By Lemma~\ref{lem:basis} there exists an orthonormal basis $B$ of $V$ so that $v^{\ast}_{c^\ast h} \geq c^\ast h^{1/2}n^{-1}$ for each $v \in B$. Write $\|P(X+w)^{\dagger} \|^2_2 = \sum_{v \in B} |\la v,X + w \ra|^2 $
and note that  
\[ \1( \|P(X+w)^{\dagger}\|_2 \leq \kappa \cdot d^{1/2}h^{3/2}n^{-3/2}) \leq \frac{2}{h} \sum_{v\in B} \1( |\la v,X+w \ra | \leq 2\kappa \cdot d^{1/2}hn^{-3/2} ), \]
since if $\|P(X+w)^{\dagger} \|_2 < \kappa \cdot d^{1/2}h^{3/2}n^{-3/2}$ then at most $h/2$
inner products on the right-hand-side are $> 2\kappa\cdot d^{1/2}hn^{-3/2}$. Taking expectations gives
\begin{align*}
\PP( \|P(X+w)^{\dagger}\|_2 \leq \kappa \cdot d^{1/2}h^{3/2}n^{-3/2}) &\leq \frac{2}{h} \sum_{v\in B} \PP( |\la v,X+w \ra | \leq 2\kappa \cdot d^{1/2}hn^{-3/2} )\\
&\le 2 \max_{v\in B}\PP( |\la v,X+w \ra | \leq 2\kappa \cdot d^{1/2}hn^{-3/2} )
\end{align*}
where we have used that $|B| = h$. We now estimate this probability by applying Lemma~\ref{lem:anti-concentration-bernoulli} with
$k = c^\ast h$, $\rho = c^\ast h^{1/2}n^{-1}$ and $r = \max\{ 2\kappa\cdot d^{1/2}hn^{-3/2}, \beta \rho/\sqrt{2} \}$ and claim that 
\[\max_{v\in B}\PP( |\la v,X+w \ra | \leq 2\kappa \cdot d^{1/2}hn^{-3/2} )\le O(\kappa + d^{-1/4})\]
as desired. To see the final inequality, note that if $r = \beta \rho/\sqrt{2}$, then the upper bound simplifies as 
\[\frac{C_{\beta}}{(kp)^{1/2}} \cdot \frac{r}{\rho} = \frac{C_{\beta} (\beta \rho/\sqrt{2})}{(c^\ast hp)^{1/2} \cdot \rho}\le O\bigg(\frac{1}{(hp)^{1/2}}\bigg) = O(d^{-3/8})\]
where we recall that $hp\ge n\cdot d^{-1/4} \cdot d/n\ge d^{3/4}$. On the other hand, if 
$r = 2\kappa\cdot d^{1/2}hn^{-3/2} = 2\kappa \cdot hn^{-1} p^{1/2}$ then coming from Lemma~\ref{lem:anti-concentration-bernoulli} simplifies as 
\[\frac{C_{\beta}}{(kp)^{1/2}} \cdot \frac{r}{\rho} = \frac{C_{\beta}}{(c^\ast hp)^{1/2}}\cdot \frac{2\kappa \cdot hn^{-1} p^{1/2}}{c^{\ast}h^{1/2}n^{-1}} = O(\kappa).\]
This gives the desired result. \end{proof}

\vspace{2mm}

\subsection{Proof of the small \texorpdfstring{$h$}{h} case}

The proof of Lemma~\ref{lem:proj-anti-concentration} is similar to the proof of Lemma~\ref{lem:hlarge-proj-anti-concentration} but we will additionally need our results that tell us the small singular vectors are unstructured.

\begin{proof}[Proof of Lemma~\ref{lem:proj-anti-concentration}]
After applying Lemma~\ref{lem:hlarge-proj-anti-concentration} and using $h\le 2\eps n$, we may assume
$1\le h < nd^{-1/4}$ here. We let $V$ be the subspace spanned by the $h$ smallest right-singular directions of $M$. As before, we apply Lemma~\ref{lem:basis} to find an orthonormal basis $B$ of $V$ so that 
\begin{equation} \label{eq:lb-on-vkast}
v_{c^\ast h}^{\ast} \geq c^\ast h^{1/2}n^{-1}
\end{equation}
for each $v \in B$. Again write  $\|P_{M,r} (X+w)^{\dagger} \|^2_2 = \sum_{v \in B} |\la v,X+w \ra|^2 $ and again note we have
\begin{equation}\label{eq:Quad-sum}
\PP( \|P_{r,M}(X+w)^{\dagger} \|_2 \leq \eta_rd^{1/2} ) \leq \frac{2}{h} \sum_{v \in B} \PP( |\la v,X +w\ra | \leq 2\eta_r(d/h)^{1/2}). \end{equation}
Now fix $v \in B$ and express $v = \sum_{i=1}^h c_iu_i$ where $u_i$ are unit vectors associated with the least right-singular directions of $M$ and $\sum_{i=1}^h|c_i|^2=1 $. We use $\sigma_{\lceil t\rceil-h+1}(A_{t,z})=\sigma_r(A_{t,z}) \leq d^{1/2}\eta_r $ to see
\begin{equation}\label{eq:near-kernel-confirm}
\|A_{t,z}v\|_2^2 = \la v , A_{t,z}^\dagger A_{t,z}v \ra = \sum_{i=1}^h\sum_{j=1}^h c_ic_j \la u_i, A^\dagger_{t,z}A_{t,z} u_j \ra =\sum_{i=1}^h |c_i|^2\s_{\lceil t\rceil-i+1}^2 \leq d\eta_r^2,
\end{equation} 
where we have used that $u_i$ are orthogonal eigenvectors of $A_{t,z}^\dagger A_{t,z}$, by definition. 
Due to \eqref{eq:lb-on-vkast} and \eqref{eq:near-kernel-confirm} we may apply Lemma~\ref{cor:spread-vector} to see 
\[ v^{\ast}_{\ell} \ge \eta_rd^{1/2}h^{-1/2} \hspace{2mm} \text{ where } \hspace{2mm} \ell = (n/(2d))\log_{(2)} d. \] 
Thus the expected intersection of the support of $X$ with the coordinates of $v$ with $|v_i| \geq \eta_r d^{1/2}h^{-1/2}$ is $p\ell = (1/2)\log_{(2)} d \rightarrow \infty$. In particular, we can use Lemma~\ref{lem:anti-concentration-bernoulli} to see
\begin{equation}\label{eq:anti-concentration-step}
\PP( |\la v,X+w \ra | \leq \eta_r(d/h)^{1/2} ) = O( (\log_{(2)} d )^{-1/2}).
\end{equation} 
with the choice of $r=\rho=\eta_rd^{1/2}h^{-1/2}$. Applying \eqref{eq:anti-concentration-step} to each term in \eqref{eq:Quad-sum}  concludes the proof of Lemma~\ref{lem:proj-anti-concentration}.
\end{proof}

\section{A step in the process}\label{sec:a-step}

The following crucial lemma tells us that each row or column addition allows us to bring a new singular value under our control, with probability $1-o_{d\rightarrow\infty}(1)$. Recall that 
\[T_{r,t}(z) = -\frac{1}{n} \sum_{j = 1}^{r} \log\big(\sigma_j\big(A_{t,z}^\ast\big)\big).\]

\begin{lemma}\label{lem:rank-jump}
For $t-2\eps n\le r<\lceil t\rceil$, let $A_t$ be a $t \times t$ matrix which, if  $\lceil t\rceil - r +1\le nd^{-1/4}$,  additionally satisfies $ A_t^\dagger\in\cE_{r+1}$, if $t \in \Z$, or $A_t \in \cE_{r+1}$, if $t$ is half integral. Then
\[ \PP\big(\,T_{r+1,t+1/2}(z) \leq T_{r,t}(z)  + \delta_{r+1} \big) =  1-O(\eps + (\log_{(2)} d)^{-1/2}), \] where the probability is over the new row or column.
\end{lemma}

To prove this we will have to track how the singular values evolve when we add a row or column to $A_t$. In particular we use the following basic linear-algebraic lemma along with Lemma~\ref{lem:proj-anti-concentration}, the main lemma from the previous section. 

\begin{lemma}\label{prop:walk-row}
Let $M$ be an $n\times m$ matrix and let $M'$ be an $(n+1)\times m$ obtained by adding the row $X$ to $M$. For $r < m$, we have 
\[ \prod_{i=1}^{r+1} \s_i(M') \geq \|P X^\dagger \|_2 \cdot \prod_{i=1}^{r} \s_i(M), \] where $P$ is the orthogonal projection onto the span of the $m-r$ smallest right-singular vector of $M$.
\end{lemma}
\begin{proof}
Let $Q$ be any $(r+1)\times (n+1)$ matrix such that $QQ^\dag = I_{r+1}$. By Courant--Fischer applied to $M'M'^{\dagger}$, we have that $s_k(QM'M'^{\dagger}Q^{\dag})\le \s_k(M')^2$ for all $k$ and therefore 
\[\det(QM'{M'}^{\dagger}Q^{\dag})\le \prod_{i=1}^{r+1} \s_i(M')^2.\]
We now choose $Q$ such that $\det(QM'M'^{\dagger}Q^{\dag})$ is exactly the RHS of the claimed inequality. Let $Q'$ be an $r\times n$ matrix with rows corresponding to biggest $r$ unit left-singular vectors of $M$. $Q$ is obtained by adding an extra row and column to $Q'$ which are all zeros, except for the bottom right entry which is $1$. It is trivial to verify using orthogonality of singular vectors that $QQ^\dag = I_{r+1}$.

Note that $QM'$ is an $(r+1)\times m$ matrix. The first $r$ rows of $QM'$ are exactly the first $r$ right-singular vectors of $M'$ with the $i$-th largest singular vector scaled by $\sigma_i(M')$; this is most easily seen by using the singular value decomposition of $M'$. The final row of $QM'$ is exactly $X$. Therefore using the base times height formula for determinants, we have that 
\[\det(QM'{M'}^\dagger Q^{\dagger}) = \on{dist}(X,\on{span}_{\C}(\{e_iQ'M\}_{1\le i\le r}))^2\prod_{i=1}^{r} \s_i(M)^2 = \|P X^\dagger \|_2^2 \cdot \prod_{i=1}^{r} \s_i(M)^2,\]
which completes the proof. 
\end{proof}

We also require Cauchy interlacing (for singular values). 

\begin{fact}\label{fact:cauchy-interlacing}
Let $M$ be an $n\times m$ matrix and let $M'$ be $M$ with a row added. Then
\[ \s_m(M) \le \s_m(M') \le \s_{m-1}(M) \le \s_{m-1}(M') \le \cdots \le \s_1(M) \le \s_1(M'). \]
\end{fact}

We now prove Lemma~\ref{lem:rank-jump}.
\begin{proof}[Proof of Lemma~\ref{lem:rank-jump}]
We only need to put the pieces together that we have already built up. Note that $A_{t,z}^{\ast} = d^{-1/2}A - zI = d^{-1/2}(A - zd^{1/2}I) = d^{-1/2}A_{t,zd^{1/2}}$. If $\sigma_{r+1}(A_{t,z}^{\ast}) \geq \eta_{r+1}$ then we have
\[ \prod_{i=1}^{r+1} \sigma_i(A_{t+1/2,z}^{\ast}) \geq 
\prod_{i=1}^{r+1} \sigma_i(A_{t,z}^{\ast}) \geq \eta_{r+1} \prod_{i=1}^{r} \s_i(A_{t,z}^{\ast}), \] 
where the first inequality holds by interlacing. Thus we are done in this case.

Otherwise $\sigma_{r+1}(A_{t,z}^{\ast}) < \eta_{r+1}$, in which case we use that $A_t \in \cE_{r+1}$. Let us focus on the situation where $t$ is half-integral and we are adding a row. The integral case is similar except we apply the relevant argument to $A_t^\dagger$, so we omit it. By Lemma~\ref{prop:walk-row},
\[ \prod_{i=1}^{r+1} \s_i(A_{t+1/2,z}^{\ast}) \geq \|P_{r+1,A_{t,z}^{\ast}}(d^{-1/2}X + w)^\dagger\|_2 \cdot \prod_{i=1}^{r} \s_i(A_{t,z}^{\ast}), \]
where $X \sim \row_{\lceil t\rceil}(\xi,p)$ is the new row added and $w$ is all $0$ except perhaps a nonzero element corresponding to $-z$ on the diagonal (when $t$ is half-integral). We may now apply Lemma~\ref{lem:proj-anti-concentration} with $r$ replaced by $r+1$ to see
\begin{align*}
\PP_X\big( \|P_{r+1,A_{t,z}^{\ast}} (d^{-1/2}X + w)^{\dagger}\|_2 < \eta_{r+1} \big)  &= \PP_X\big( \|P_{r+1,A_{t,zd^{1/2}}} (X + d^{1/2}w)^{\dagger}\|_2 < d^{1/2}\eta_{r+1} \big) \\
&= O(\eps + (\log_{(2)} d )^{-1/2})   
\end{align*}
as desired.
\end{proof}

\section{Analysis of the process}\label{sec:analysis}
Recall that we index time in half-integer steps from the interval $[m,n]$ so that at each integer time $t$ we have that $A_t$ is a $t\times t$ matrix and at time $t+1/2$, $A_{t+1/2}$ is the $t \times (t+1)$ matrix which is $A_{t+1}$ with the bottom-most row deleted. Recall $r(t)$ can be defined by setting $r(m)=(1-\eps/4)m$ and
\[ r(t+1/2) = \begin{cases}\, r(t)+1 &\qquad \text{ if }  T_{r(t)+1,t+1/2}( z ) \leq T_{r(t),t}(z) + \delta_{r(t)+1} ;\\
\, r(t) &\qquad \text{ otherwise.} \end{cases}\] 
In what follows, if $M$ is an $n\times m$ matrix, we define $\sigma_r(M) = 0 $ if $r > m$.
We track the evolution
of the random variable $h(t) = t - r(t)$. The main goal of this section is to show that, whp, we have taken \emph{all} singular values into our sum by the end of the process.

\begin{lemma}\label{lem:main-height-lemma} 
For all $z \in \C \setminus \{0\}$, we have 
\[ \PP( h(n) = 0 ) = 1-o(1). \]
\end{lemma}

We shall also see that this immediately implies the following result, which is the ``hard'' direction of our two-sided comparison of $T_n(z)$ and $U_n(z)$. 

\begin{lemma}\label{lem:U-at-most-T}
For all $z \in \C \setminus \{0\}$, we have 
$U_n(z) \leq T_n(z) + o(1)$ with probability $1-o(1)$.
\end{lemma}

\noindent To prove these lemmas we study how $h(t)$ evolves. For convenience when analyzing cases $h(t)=0$, however, we define the slightly modified function
\[ h^\ast(t):=\begin{cases}h(t)&\text{if }h(t)\neq 1/2;\\0&\text{if }h(t)=1/2.\end{cases} \]
At the start of the process we have 
\[ h^\ast(m) \leq m - (1-\eps/4)m = \eps m/4 < \eps n/4, \]
by definition. For all $t < n$, we have 
\[ h^\ast(t+1/2) \leq h^\ast(t)+1/2+(1/2)\1_{h^\ast(t)=0}, \]
and finally, if $A_{t}^\dagger \in \cE_{r(t)+1}$ for $t$ integral or $A_{t} \in \cE_{r(t)+1}$ for $t$ half-integral, we know from Lemma~\ref{lem:rank-jump} that  
\[ \PP\big( h^\ast(t+1/2) \leq h^\ast(t) - (1/2)\1_{h^\ast(t) > 0} \big) = 1-o_{\substack{d\rightarrow\infty\\\eps\to0}}(1). \]
Of course, for this to be useful, we need to guarantee that the required event holds sufficiently often. One difficulty here is that our bounds on the failure of $\cE_{r}$ are not sufficiently strong to ensure that our matrix always
satisfies the appropriate condition. To get around this we use a simple idea: if at time $t$ we have $h^\ast(t) < \lfloor (n-t)/8 \rfloor $ then we don't worry about certifying that new singular values are taken into our sum, as we are already ``over-achieving'' at such a time. Thus we only need to union bound over all pairs $(t,r)$ where $r \leq t-\lfloor (n-t)/8 \rfloor$. Thus, for all times in $(1/2)\Z$, we define 
\[ \cQ_{t} =  \bigcap_{r=t - n/d^{1/4}}^{t-\lfloor (n-t)/8 \rfloor }   \{ A_t \in \cE_{r} \}, \quad \text{if }t \in  \big((1/2)\Z \big) \setminus \Z \quad \text{ and } \quad \cQ_{t} = \bigcap_{r=t - n/d^{1/4}}^{t-\lfloor (n-t)/8 \rfloor }   \{ A_t^\dagger \in \cE_{r} \}, \quad \text{ if }t \in \Z. \]
We then set \[ \wt{\cQ}_t = \bigcap_{t'\le t} \cQ_{t} , \]
where the latter intersection is over all $t' \in (1/2)\Z \cap [m,t]$. The key point here is that $\wt{\cQ}$ holds with high probability. 

\begin{lemma}\label{lem:its-union-bound-time}
$\PP(\wt{\cQ}_n) = 1-o(1)$.
\end{lemma}
\begin{proof}
Since each of the events in the definition of $\cQ_t$ are of the form $A_t \in \cE_r$ where $r \geq t-nd^{-1/2}$, we may apply Lemma~\ref{lem:quasi-random} to bound $\PP\big( \wt{\cQ}^c \big)$ above by
\[ \sum_{t=m}^{n}\sum_{r=t - n/d^{1/4}}^{t-\lfloor (n-t)/8\rfloor} \PP( A_t \not\in \cE_r ) \leq n^{-1} + \sum_{t=m}^{n}\sum_{r=t - n/d^{1/4}}^{t-\lfloor (n-t)/8\rfloor} \exp\big(-d^{1/2}(\lceil t\rceil-r+1)\big). \]
This is 
\[ \hspace{-2em} \leq n^{-1} + 2\sum_{t=m}^{n}\exp\big(-d^{1/2}((n-t)/8 + 2\big) \leq n^{-1} + 2\sum_{k\ge 1}\exp\big(-kd^{1/2}/8\big),\]
which tends to zero as $n$ and $d$ tend to infinity. 
\end{proof}

\vspace{2mm}

To avoid the property $\wt{\cQ}_n$ running interference with the independence in the row/column revelation process, we couple $h^\ast(t)$ to a simpler process. For this let $\cF_t$ be the $\s$-algebra corresponding to the matrix $A_t$. Now define the random variables $X_t=h^\ast(t)\1_{\wt{\cQ}_t}$.
This definition is crafted so that on $\wt{\cQ}_n$, \emph{if $X_n=0$ then $h^\ast(n) = 0$}. We now prove the following simple probabilistic lemma which shows that $X_t$ has sufficient downward drift to ensure $X_n = 0$ with high probability. Such lemmas originate in the work of Costello, Tao, and Vu \cite{CTV06} on singularity of symmetric random matrices, and have been used more recently to study singularity and rank in sparse random matrices \cite{FKSS23,GKSS23}.

\begin{lemma}\label{lem:random-walk}
Let $(\cF_s)_{s=0}^T$ be a filtration and let $(Y_s)_{s = 0}^T$ be a sequence of random variables for which $Y_s$ is $\cF_s$-measurable, $Y_s \in (1/2)\Z_{\geq 0}$, $Y_0\le T/8$, $Y_{s+1}\le Y_s+1$, and 
such that 
\[ \mb{P}(Y_{s+1}\le Y_s-(1/2)\1_{Y_s > 0} \, | \cF_s) \ge 1-q,\]
whenever $Y_s\ge\lfloor(T-s)/16\rfloor$. Then 
\begin{equation} \label{eq:random-walk-conclusion}
\mb{P}(Y_T = 0) \geq 1-4q^{1/8}.
\end{equation}
\end{lemma}
\begin{proof}
We proceed by considering an appropriate exponential moment: define the random variables $Z_s = q^{(T-s)/16}q^{-Y_s/2}$ and note 
\[ \PP(Y_T\geq 1/2) = \PP( Z_T \geq q^{-1/4} )  \leq q^{1/4} \cdot \EE Z_T. \]
We now show that $\EE Z_T \leq 4$ by bounding how much the expectation moves in each step. Indeed, for each $s \leq T-1$, we claim we have 
\begin{equation}\label{eq:E-movement}
\EE[ Z_{s+1}\, |\, \cF_s\, ] \leq 1 + 2q^{1/8}Z_{s}.
\end{equation}
To see this, first consider the case $Y_s < \lfloor (T-s)/16 \rfloor$; this forces $T-s\ge 16$ as $Y_s\ge 0$. Then $Y_{s+1} \leq \lfloor (T-s)/16 \rfloor+1/2$ and thus 
\[Z_{s+1} \le q^{(T-s-1)/16 - 1/2\lfloor (T-s)/16 \rfloor-1/4 } \leq 1.\] Otherwise we have $Y_s \geq \lfloor (T-s)/16 \rfloor$ and we calculate 
\[ \EE[ Z_{s+1}\, |\, \cF_{s}\, ] \leq q^{(T-s-1)/16}\big( q^{-1/2(Y_s-(1/2)\1_{Y_s>0})} + q \cdot q^{-1/2(Y_s+1)}\big) \leq 1 + 2q^{1/8}Z_{s}. \]
This establishes \eqref{eq:E-movement}. We now iteratively apply \eqref{eq:E-movement} to see 
\begin{align*}
\mb{E}[Z_T]&\le 1+2q^{1/8}+(2q^{1/8})^2+\cdots+(2q^{1/8})^{T-1}+(2q^{1/8})^T\mb{E}[Z_0]\le 2 + (2q^{1/8})^{T}\le 4.
\end{align*}
For the inequality we used the fact we may assume $q \le 1/2^{16}$, otherwise the conclusion at \eqref{eq:random-walk-conclusion} is trivial.
\end{proof}

\begin{proof}[Proof of Lemma~\ref{lem:main-height-lemma}] 
All that remains is to check that the pieces fit together. Note that $X_t$ is $\cF_t$-measurable, by definition. Let $Y_s=X_{m+s/2}$ for $s\in[0,2(n-m)]\cap\Z$. We have $T = 2\eps n$ for our process and our starting point $Y_0=X_m$ satisfies $Y_0\le \eps m/4 \leq \eps n/4 = T/8$, by definition. We now claim that
\begin{equation}\label{eq:walk-check}
\PP\big( X_{t+1/2}  \leq  X_t -(1/2)\1_{X_t > 0}\, |\, \cF_{t} \big)  \geq 1- o_{d\rightarrow \infty}(1),
\end{equation}
whenever $X_t \geq \lfloor (n-t)/8 \rfloor$. We check this inequality pointwise. If $A_t \in \cQ_t^c$ then \eqref{eq:walk-check} holds, by definition, since then $X_{t+1/2}=X_t=0$. On the other hand, if $A_t \in \cQ_t$ then we apply Lemma~\ref{lem:rank-jump} to see that 
\[ \PP\big( X_{t+1/2} \leq X_t - (1/2)\1_{X_t>0}\, |\, \cQ_t \big) 
= \PP\big( h^\ast(t+1/2)  \leq h^\ast(t)-(1/2)\1_{h^\ast(t) > 0}\, |\, \cQ_t \big) \geq 1 - o_{\substack{d\rightarrow\infty\\\eps\to0}}(1).\]
Thus $X_t$ is a random process that satisfies the hypothesis of Lemma~\ref{lem:random-walk}. We apply Lemma~\ref{lem:its-union-bound-time} and then Lemma~\ref{lem:random-walk} to see that
\[ \PP\big( h^\ast(n) > 0 \big) = \PP\big( h^\ast(n) > 0 \cap \wt{\cQ}_n \big) + o(1) \le \PP\big( X_n > 0 \big)  + o(1)  = o(1).\]
Using that $h^\ast(n)=0$ implies $h(n)=0$ (since $n$ is integral), we are done.
\end{proof}

\vspace{2mm}

We now deduce the important consequence of Lemma~\ref{lem:main-height-lemma}, Lemma~\ref{lem:U-at-most-T}, which says that $U_n(z) \leq T_n(z) + o(1)$, for all $z \not= 0$.

\begin{proof}[Proof of Lemma~\ref{lem:U-at-most-T}]
Let us fix a sequence of matrices for which $h(n)=0$. We now argue deterministically that $U_n(z) \leq T_n(z) +o(1)$. Since $h(n)=0$ holds with probability $1-o(1)$, by Lemma~\ref{lem:main-height-lemma}, we will conclude that $U_n(z) \leq T_n(z) +o(1)$ holds with probability $1-o(1)$ and thus conclude the proof of the lemma.

Now, since $h(n)=0$, for each $(1-\eps/4)m \leq r \leq n$, there is $t = t(r) \in [m,n] \cap \frac{1}{2}\Z$ so that  
\begin{equation}\label{eq:heightlem-1} T_{r,t(r)}( z ) \leq T_{r-1,t(r)-1/2}(z)  + \delta_{r} \end{equation}
and such that $t(\cdot)$ is a strictly increasing function. Additionally, by Fact~\ref{fact:cauchy-interlacing}, we have
\begin{equation}\label{eq:heightlem-2} T_{r,t}(z)\le T_{r,t-1/2}(z), \end{equation}
for all $t$. Chaining \eqref{eq:heightlem-1} and \eqref{eq:heightlem-2} together gives 
\[\,U_{n}(z) = T_{n,n}(z) \leq T_{m,(1-\eps/4)n}(z)\hspace{1mm} + \sum_{r= (1-\eps/4)m}^n \delta_r  = T_n(z) + o(1), \]
where the last equality holds since $\sum_r \delta_r = O(\eps)$, as noted at \eqref{eq:def-delta_r}.
Thus $U_n(z) \leq T_n(z) +o(1)$ under the assumption $h(n)=0$. Since $h(n)=0$ with probability $1-o(1)$, by Lemma~\ref{lem:main-height-lemma}, we conclude the proof of the lemma.
\end{proof}

\section{Completion of the proof}\label{sec:completion}
With the main difficulty of the proof behind us, we no longer need to consider the sequence $A_t$, and now only require $A_n$ and $A_m$. It is convenient to work with a modified version of $T_n(z)$ which is a function of $A_{n,z}^{\ast}$ instead of $A_{m,z}^{\ast}$. We define 
\[ T_n^{(1)}(z) = -\frac{1}{n}\sum_{i=1}^{(1-\eps/4)m}\log(\sigma_i(A_{n,z}^{\ast})). \]
We also need a related function that \emph{almost} does not depend on $A_m$
\[ T_n^{(2)}(z) = -\frac{1}{n}\sum_{i=2\eps n}^{(1-\eps)n}\log(\sigma_i(A_{n,z}^{\ast})) - \frac{(1-\eps/4)m - (1-3\eps)n}{n}\log(\sigma_{m(1- \eps/4)}(A_{m,z}^{\ast})).\]
We now relate $U_n(z)$, $T_n(z)$, $T_n^{(1)}(z)$, and $T_n^{(2)}(z)$. Recall that 
\[ T_n(z) = -\frac{1}{n}\sum_{i=1}^{(1-\eps/4)m} \log\big( \sigma_i(A_{m,z}^{\ast}) \big), \quad U_n(z) = - \frac{1}{n}\sum_{i=1}^{(1-\eps/4)m}  \log\big(\sigma_i(A_{n,z}^{\ast})\big). \]

\begin{fact}\label{fact:det-facts}
We have
\[
    T_n(z) \le T_n^{(2)}(z)  \qquad \text{ and } \qquad U_n(z) \ge\big( (1-\eps/4)(1-\eps) \big)^{-1} T_n^{(1)}(z) 
\]
\end{fact}
\begin{proof}
The second inequality follows from the fact that $\sigma_1(A_{n,z}^{\ast})\ge \sigma_2(A_{n,z}^{\ast})\ge\cdots \ge \sigma_n(A_{n,z}^{\ast})$. For the first inequality, we again use the ordering of the singular values to write
\[ -\frac{1}{n}\sum_{i=1}^{(1-\eps/4)m} \log\big( \sigma_i(A_{m,z}^{\ast}) \big) \le -\frac{1}{n}\sum_{i=1}^{m-3\eps n} \log\big( \sigma_i(A_{m,z}^{\ast}) \big) - \frac{(1-\eps/4)m - (1-3\eps)n}{n}\log(\sigma_{(1-\eps/4)m}(A_{m,z}^{\ast}))\]
We now apply the interlacing inequality for singular values (Fact~\ref{fact:cauchy-interlacing}) $2\eps n$ times to see that the above is at most $T^{(2)}(z)$, as desired.
\end{proof}

\vspace{2mm}

Next we will require control on the Hilbert--Schmidt norm of $A_{n,z}^{\ast}$.
\begin{fact}\label{fact:HS-norm}
With high probability, as $n\to \infty$, we have
\[\|A_{n,z}^\ast\|_\mr{HS}^2 \le 4(|z|^2 + 1)n.\]
\end{fact}
\begin{proof}
Note that 
\[\mb{E}\, \|A_{n,z}^\ast\|_\mr{HS}^2 = \mb{E}\, \| d^{-1/2}A_{n}-zI_n\|_\mr{HS}^2\le 2(\mb{E} \| d^{-1/2}A_{n}\|_\mr{HS}^2 + \|zI_n\|_\mr{HS}^2) = 2(1 + |z|^2)n.\]
Furthermore note that the square of the Hilbert--Schmidt norm as the entry-wise sum of the squares. Therefore by the Chernoff inequality we have that whp at most $2dn$ of the variables $\delta_{i,j}$ are nonzero and result follows via the strong law of large numbers applied to $\xi_{i,j}$ where $\delta_{i,j}$ are nonzero. 
\end{proof}

\vspace{2mm}

Finally we require the convergence of the truncated log potentials $T_n^{(1)}$,$T_n^{(2)}$. Since the proof of this is fairly standard and of a slightly different flavor to the proof here, we treat it in Section~\ref{sec:lem:conv-of-Tn}.

\begin{lemma}\label{lem:conv-Of-Tn}
For $pn \rightarrow \infty$ and $p\to 0$, let $A_n \sim \Delta_{n}(\xi,p)$ and let $\eps \rightarrow 0$ sufficiently slowly. Then 
$T_n^{(1)}(z)$ and $T_n^{(2)}(z)$ converge to $U^{\circ}(z)$ in probability. 
\end{lemma}

To finish the proof of our main theorem we state a simple criterion that allows us conclude the convergence of the spectral measures from the point-wise convergence of the $U_n(z)$. The following lemma is a simple reworking of Theorem 2.1 in \cite{TV10}.

\begin{prop}\label{prop:unicity}
For each $n$, let $A_n$ be a random $n\times n$ matrix and let $\mu_n = \mu_{d^{-1/2}A_n}$ be the scaled spectral law of $A_n$. If $\EE\,  \|d^{-1/2}A_n\|_{\mr{HS}} = O(n^{1/2})$ and, for almost all $z \in \C$, we have that $U_{\mu_n}(z)$ converges to $U^{\circ}(z)$ in probability, then $\mu_n \rightarrow \mu^{\circ}$ in probability.
\end{prop}

We now prove Theorem~\ref{thm:main}, assuming Lemma~\ref{lem:conv-Of-Tn}.

\begin{proof}[Proof of Theorem~\ref{thm:main-complex}]
Combining Fact~\ref{fact:det-facts}, Lemma \ref{lem:U-at-most-T}, and Lemma~\ref{lem:conv-Of-Tn}, we see 
\[\mb{P}(U_n(z)\le U^\circ(z) + o(1)) = 1-o(1).\]
Combining the second item of Fact~\ref{fact:det-facts} and Lemma~\ref{lem:conv-Of-Tn}, we have 
\[\mb{P}(U_n(z)\ge U^\circ(z) -o(1)) = 1-o(1).\]
Therefore, Fact~\ref{fact:HS-norm} allows us to invoke Proposition~\ref{prop:unicity}, which completes the proof.
\end{proof}

\section{Proof of Lemma~\ref{lem:conv-Of-Tn}: convergence of $T_n^{(1)}$, $T_n^{(2)}$}\label{sec:lem:conv-of-Tn}
We have now completed the proof of our main theorem, Theorem~\ref{thm:main}, modulo the proof of Lemma~\ref{lem:conv-Of-Tn} which we have deferred to this later section, since it is neither particularly difficult nor particularly original. Indeed, working with the sums $T_n^{(1)}(z)$, $T_n^{(2)}(z)$ is much easier than working with $U(z)$ since we easily rule out the possibility that any one term in these sums is large. 

 Define $\nu_{n,z}$ to be the empirical measure of all of the singular values of the shifted matrix $A_{n,z}^\ast$,
 \[ \nu_{n,z} = \frac{1}{n}\sum_{i} \delta_{\sigma_i(A_{n,z}^{\ast})} \] and define $G_n$ to be an iid $n\times n$ random matrix with entries distributed as variance $1$ complex Gaussians: $\frac{1}{\sqrt{2}}(Z_1 + Z_2i)$, with $Z_1,Z_2 \sim N(0,1)$ iid. Let $\nu_{n,z}^{G}$ be the empirical measure of the singular values of the shifted matrix $G_{n,z}=n^{-1/2}G_n-zI_n$. It is well known (see Appendix A of \cite{coo19}) that $\nu^{G}_{n,z}$ converges in probability to a limit 
$\nu^{G}_{z}$, which satisfies the following. 

\begin{fact}\label{fact:Gaussian-fact}
There exists an absolute constant $C>0$ such that the following holds. For all $z \in \C$, $\nu^G_{z}$ is a measure on $[0,|z|+C]$ with $\nu_z([0,t)) \le Ct$ for all $t$, and 
\[ U^{\circ}(z) =  -\int\log t\,  d\nu_z^G (z). \] 
\end{fact}

Next we establish the convergence of the measures $\nu_{n,z}$. We do this by following a now standard method; we truncate the entries of the matrix, to ensure we have bounded moments, and then apply the trace-moment method to the truncated matrix. We only include a sketch of this here and direct the reader in search of a complete proof to \cite[Section~10.3]{RT19} or \cite[Section~9]{BR19}.

\begin{fact}\label{fact:trace-moments}
For $pn \rightarrow \infty$ and $p\to 0$, let $A_n \sim \Delta_{n}(\xi,p)$. For all bounded continuous functions $\phi \colon \R \rightarrow \R$, we have 
\[ \int \phi(t)\,  d\nu_{n,z} \to \int \phi(t)\, d\nu^{G}_{z}, \]
in probability.
\end{fact}
\begin{proof}[Proof sketch]
For the truncation step, let $M\ge 1$, let 
\[ \xi_{i,j}^{(1)} = \xi_{i,j}\1_{|\xi_{i,j}|\le M} \quad \text{ and } \quad \xi_{i,j}^{(2)}=\xi_{i,j}\1_{|\xi_{i,j}|> M}, \quad \text{so that } \quad \xi_{i,j} = \xi_{i,j}^{(1)}  + \xi_{i,j}^{(2)}.\] Assume that $M$ is large enough so that $\xi_{i,j}^{(1)}$ is non-constant. Let $A_n^{(k)}$ have entries $\delta_{i,j}\xi_{i,j}^{(k)}$ for $k\in \{1,2\}$. 
Furthermore, since the entries of $A_n^{(1)}$ are bounded by $M$, all moments of the entries exist and thefore we are able to apply a standard trace method argument: in particular we set 
\[ B_{n,z} = \big( pn\Var(|\xi_{i,j}^{(1)}|)\big)^{-1/2}\big( A_n^{(1)}-\mb{E}A_n^{(1)}\big ) - zI_n \]
and let $\nu^{(1)}_{n,z}$ be the empirical measure of the singular values of $B_{n,z}$. 
Using a standard fact from linear algebra we express
\[  \int x^{2k} d\nu^{(1)}_{n,z} =   n^{-1} \tr\big(  (B_{n,z}B_{n,z}^{\dagger})^k \big)\]
and, since all moments of $B_{n,z}$ are bounded, one can explicitly compute that, for all $k\geq 1$,
\[\lim_{n \rightarrow \infty}  n^{-1} \EE\, \tr\big(  (B_{n,z}B_{n,z}^{\dagger})^k\big) =   \int x^{2k} ~d\nu^{G}_{z}.\]
This implies, by a standard fact, that $\nu_{n,z}^{(1)}$ converges to $\nu_z^{G}$ in probability. 

To prove that $\nu_{n,z}$ converges to $\nu_z^G$ in probability, we now need to show that adding $A^{(2)}_n$ back does not disrupt the singular values too much. By the Chernoff inequality and the strong law of large numbers, we have 
\[ \PP\big( \snorm{A_n^{(2)}}_{\mr{HS}}^2\le 2pn^2\cdot \EE|\xi_{i,j}^{(2)}|^2 \big) = 1-o(1),\]
where $\EE |\xi_{i,j}|^2 \rightarrow 0 $ as we send $M\rightarrow \infty$. Thus  $\snorm{A_n^{(2)}}_{\mr{HS}} = o_{M\rightarrow \infty}\big( \snorm{A_n^{(1)}}_{\mr{HS}}\big)$ which means that the singular values of $A_n^{(1)}$ are not affected much by adding $A_n^{(2)}$. This can be made rigorous by using the Hoffman--Wielandt inequality \cite[Theorem~1]{HW53}. Finally, taking $M\to\infty$, the desired result follows. 
\end{proof}

\vspace{2mm}

We are now ready to prove Lemma~\ref{lem:conv-Of-Tn}, which we do in the course of three lemmas. The proof that $T_n^{(1)}(z)$ and $T_n^{(2)}(z)$  converge in probability to $U^{\circ}(z)$ are very similar and so we limit ourselves here to proving the result for $T_n^{(1)}(z)$. For this, we write $\eta = 5\eps/4 - \eps^2/4$ and note 
\[T_n^{(1)}(z) = -\frac{1}{n}\sum_{i=1}^{(1-\eta)n}\log(\s_i(A_{n,z}^{\ast})) + o(1),\]
with high probability, where the error term arises due to rounding $(1-\eta)n$. Indeed, the top index can be absorbed using Fact~\ref{fact:Gaussian-fact} and the convergence in distribution of $\s_i(A_{n,z}^{\ast})$.

We prove Lemma~\ref{lem:conv-Of-Tn} in three fairly straightforward steps.
For this, we define the empirical measure of the singular values appearing in the sum $T_n^{(1)}$. That is,  
\[ \nu^{(1)}_{n,z} = \frac{1}{(1-\eta)n}\sum_{i=1}^{(1-\eta)n} \delta_{\sigma_i(A^{\ast}_{n,z})} .\] In the first step we identify a
measure $\nu_z^{G,\eta}$ with the property that $\nu^{(1)}_{n,z}$ tends weakly to $\nu_z^{G,\eta}$. We then observe, in a second step, that we can upgrade this to the convergence
\begin{equation}\label{eq:conv-Tn-1} \lim_{n\rightarrow \infty} T^{(1)}_n(z) = -(1-\eta)\int \log t~d\nu_{z}^{G,\eta} + o_{\eta \rightarrow 0}(1),\end{equation} with high probability. In the third step we show the (deterministic) convergence 
\begin{equation}\label{eq:conv-Tn-2} \lim_{\eta \rightarrow 0} -\int \log t~d\nu_{z}^{G,\eta} =  -\int \log t~ d\nu_z^{G} .\end{equation}
Putting these steps together completes the proof. The following lemma is the first step.
\begin{lemma}\label{lem:trunc-conv-in-prob} There exists a measure $\nu_z^{G,\eta}$
with the property that $\nu_{n,z}^{(1)}$ converges to  $\nu_z^{G,\eta}$, in probability. 
\end{lemma}
\begin{proof}
To define the measure $\nu_z^{G,\eta}$, let $X$ be sampled according to $\nu_z^G$, and let $\tau = \tau_{z}(\eta)$ denote the unique real number for which  
\[  \mb{P}_{X}\big( X\le \tau \big)\ge \eta \hspace{4em} \text{ and } \hspace{4em} \mb{P}_{X}\big( X<\tau\big)\le \eta.\] We then set $p = (1-\eta)^{-1}\mb{P}_{X}\big( X<\tau\big)$.
We define our probability measure by defining a random variable $Y \sim \nu_z^{G,z}$. Let $Z \sim \Ber(p)$ be a random variable independent of $X$ and define
\[ Y =  Z \big( X\vert \{ X > \tau \} \big) + \tau(1-Z), \] where $\big( X\vert \{ X > \tau \} \big)$ denotes $X$ conditioned on being $>\tau$. One can now see, by Fact~\ref{fact:trace-moments} and some manipulation, that $\nu^{(1)}_{n,z}$ converges to $\nu_{z}^{G,\nu}$, in distribution. \end{proof} 

\vspace{2mm}

We now upgrade this convergence in probability to \eqref{eq:conv-Tn-1}.

\begin{lemma}\label{lem:trunc-conv-log} We have 
\[ \lim_{n\rightarrow \infty} T^{(1)}_n(z) = -(1-\eta)\int \log t~d\nu_{z}^{G,\eta} + o_{\eta \rightarrow 0}(1) .\]
\end{lemma}
\begin{proof}
We define a smooth bump function $\psi$ on the support of 
$\nu_{n,z}^{(1)}$ for which the function $\psi( \cdot ) \log( \cdot )$ is bounded. Let 
$\psi\colon\mb{R}\to\mb{R}$ be a smooth function with $\|\psi\|_{\infty} \leq 1$, $\psi \geq 0$, and with 
\[ \psi(x) = \begin{cases} \, 1 \qquad \text{ for }  \qquad   \eta/(2C) \leq x \leq|z| + \eta^{-1} ; \\  
\, 0 \qquad \text{ for } \qquad  x < \eta/(4C); \\
\, 0  \qquad \text{ for } \qquad x > |z| + 2\eta^{-1}. \end{cases}, \] where $C$ is as in Fact~\ref{fact:Gaussian-fact}. We now use $\psi$ to approximate $T_n^{(1)}$ by writing
\begin{equation} \label{eq:approx-T} \bigg|-\frac{1}{n}\sum_{i = 1}^{(1-\eta)n}\psi(\sigma_i(A_{n,z}^{\ast}))\log(\sigma_i(A_{n,z}^{\ast})) -  T_n^{(1)}(z) \bigg|\leq \frac{1}{n}\sum_i \1\big( \psi(\sigma_i) <1  \big) \big|\log \sigma_i\big|.  \end{equation}
By Fact~\ref{fact:Gaussian-fact} and Fact~\ref{fact:trace-moments} we can deal with the small $\sigma_i$ on the right of \eqref{eq:approx-T} using that 
\[ \PP\big( \sigma_{(1-\eta)n}(A^{\ast}_{n,z}) \geq \eta/(2C) \big) = 1-o(1). \]
For the large $\sigma_i$ note $\log( |x| +1)\1( |x| > M) \leq x^2/M$, so we have that \eqref{eq:approx-T}
\[\frac{1}{n}\sum_i \1\big( \psi(\sigma_i) <1  \big) \big|\log \sigma_i \big|\leq \frac{\eta}{n} \sum_i \sigma_i^2 \leq \frac{\eta}{n}\snorm{A_{n,z}^{\ast}}_{\mr{HS}}^2\le \eta^{1/2},  \]
with high probability. Here, the last inequality holds by Fact~\ref{fact:HS-norm}. 

Now, since  $\psi( \cdot ) \log( \cdot )$ is bounded, convergence in probability (Lemma~\ref{lem:trunc-conv-in-prob}) yields
\[
-\frac{1}{(1-\eta)n}\sum_{i = 1}^{(1-\eta)n}\psi(\sigma_i(A_{n,z}^{\ast}))\log(\sigma_i(A_{n,z}^{\ast}))= -\int \psi(t)\log t~ d\nu_{n,z}^{(1)} = -\int \psi(t)\log t~d\nu_{z}^{G,\eta}+o_{n\rightarrow \infty}(1),\]
with high probability. Finally note that 
\[ -\int \psi(t)\log t~d\nu_{z}^{G,\eta}= -\int \log t~d\nu_{z}^{G,\eta}, \]
since $\nu_{z}^{G,\eta}([0,\eta/(2C)]) \leq \nu_{z}^{G,\eta}([0,\tau/2]) = 0 $, by Fact~\ref{fact:Gaussian-fact} and the definition of $\nu_{z}^{G,\eta}$.
\end{proof}

\vspace{2mm}

We now prove the following which is the third and final step in our proof of Lemma~\ref{lem:conv-Of-Tn}. 

\begin{lemma}\label{lem:limits-converge}
\[ \lim_{\eta \rightarrow 0} -\int \log t~d\nu_{z}^{G,\eta} =  -\int \log t~ d\nu_z^{G} \]
\end{lemma}
\begin{proof}
Let $X \sim \nu_z^{G,\eta}$. Note that by the definition of $\nu_z^{G,\eta}$, we have 
\[ -\int \log t~d\nu_{z}^{G,\eta} =   -\frac{1-\eta}{\PP( X > \tau)} \int_{> \tau } \log t~ d\nu_z^G  + \eta \log \tau^{-1} =  -\int \log t~ d\nu_z^{G}  + o_{\eta \rightarrow 0}(1),\]
where the last inequality holds by Fact~\ref{fact:Gaussian-fact} put together with the definition of $\tau$, which gives 
$C\tau \geq \PP( X \leq \tau ) \geq \eta $.\end{proof}

\vspace{2mm}

The proof of Lemma~\ref{lem:conv-Of-Tn} is now easy.
\begin{proof}[Proof of Lemma~\ref{lem:conv-Of-Tn}]
We simply string together Lemma~\ref{lem:trunc-conv-in-prob}, Lemma~\ref{lem:trunc-conv-log} and Lemma~\ref{lem:limits-converge} to finish.\end{proof}

\appendix

\section{Omitted proofs}\label{sec:facts}

\begin{proof}[Proof of Lemma~\ref{lem:unstructured-graph}]
Assume $t-1/2$ is integral so that $A_t$ is a $t'\times(t'+1)$ matrix where $t'=t-1/2$. The other case is strictly simpler so we omit it. We fix a set $S \subseteq [t'+1]$ of columns of size $k$. We note that 
\[ |U(S)| = \sum_{i \in [t']} \1(i \in U(S) ), \]
where the sum is over the rows and thus is a sum of independent random variables. Thus we compute the expectation and then control the deviations. For $i\in[t']$ we have
\begin{equation} \label{eq:NU-notS}
\PP( i \in U(S)  ) = (1-p)^k \text{ if }i \in S\text{ and } \PP( i \in U(S) )=\PP( |\xi| \geq \beta ) (1-p)^{k-1}pk \text{ if }i \not\in S.
\end{equation}

If $k$ is such that $ n/(2d) \leq k \leq (n/(2d))\log_{(2)} d$ we can just consider the contribution of $i \not\in S$ and then use Chernoff. Indeed, picking up from \eqref{eq:NU-notS} we have 
\[ pk(1-p)^{k-1}\ge (1/3)(1-p)^{(1/(2p)) \log_{(2)} d}\ge (1/3)\exp(-\log_{(2)} d)\ge (\log d)^{-3/2},\]
where we have used that $d$ is large and that $1-x\ge e^{-2x}$ for $x\le 1/2$. Thus 
\[ \EE |U(S)| \geq (t'-k) (\log d)^{-3/2} \geq 2(\log (n/k))^{-2} dk=2\alpha(k)dk.\]
This follows via noting that $(t'-k)\ge m-k\ge n/2$ and that $1/2\le dk/n\le \log_{(2)}d/2$. Thus by Chernoff, we have 
\begin{equation}\label{eq:UN-chernoff} 
\PP\big( |U(S)| \leq \alpha(k)dk \big) \leq \PP\big( |U(S)| \leq (1/2)\EE |U(S)| \big)  \leq \exp(-cn/(\log d)^{3/2}).
\end{equation} 
We finish this range of $k$ by simply union bounding over all subsets of size $k$: the probability there is a set $S$ of size $k$ with $|U(S)|< \alpha(k)dk$ in this range of $k$ is at most
\[ e^{-cn(\log d)^{-3/2}}\binom{n}{k} 
\leq \exp\big( - cn/(\log d)^{3/2} + (n/(2d))(\log d)\log_{(2)} d \big) \leq e^{-d^{2/3}k},   \]
where the last inequality follows since $k\le (n/(2d)) \log_{(2)}d$ and $d$ is sufficiently large.

We now turn to the trickier range $c^\ast(t-r+1)\le k\le n/(2d)$. For this, define $T = \alpha(k)dk$
and write 
\begin{equation}\label{eq:k-small-set-up}
\PP( \exists S, |S|=k, |U(S)| < T ) \leq  \binom{n}{k}\PP( |U(S)\cap S|< T )\cdot \PP( |U(S) \setminus S|< T).
\end{equation}
For the sake of notation we also set $\gamma = \PP( |\xi| \geq \beta )$ and then define 
\[ q_1 = (1-p)^k\text{ and } q_2 = \gamma(1-p)^{k-1}pk.\]
Note that $1-q_1 \leq pk \text{ and } 1-q_2 \leq 1-\gamma kp/4$. Let $k'=|S\cap[t']|$; note that $k'\in\{k-1,k\}$. Now $|U(S)\cap S|$ is distributed as binomial random variable $B(k',q_1)$ and $|U(S)\setminus S|$ is distributed as $B(t'-k',q_2)$. Therefore \eqref{eq:k-small-set-up} is at most
\[\sum_{i,j<T} \binom{n}{k'}\binom{k'}{i}(1-q_1)^{k'-i} \cdot
\binom{t'-k'}{j}(1-q_2)^{t'-k'-j} \] 
\[ \leq (T+1)^2\bigg(\frac{en}{k'}\bigg)^{k'}2^{k'}(k'p)^{k'-T}\bigg(\frac{en}{T}\bigg)^{T}(1-\g kp/4)^{n/2},\]
where we used that $k'+T < n/3$ and bounds on $q_1,q_2$. Recall that $\alpha(x) = (\log(n/x))^{-2}$ so $T = dk(\log(n/k))^{-2}$. Thus we bound the above by 
\begin{align*}
 (4epnd)^{k'}\bigg(\frac{2en^2}{kdT}\bigg)^{T}e^{-\gamma kpn/8}&\le(4ed^2)^{k}(n/k)^{5T}\exp(-\gamma kd/8)\\  
 &\le(2epn)^ke^{5dk/\log(n/k)}\exp(-\gamma kd/8) \\
 = e^{k(\log(2ed) + 5d/\log(n/k) - \gamma d/8)}\le e^{-d^{2/3}k};
\end{align*}
in the final line we use that $\log(n/k)\ge \log(2d)$ and that $d$ is larger than and absolute constant. The desired result follows after summing over all $k\ge c^\ast(t-r+1)$.
\end{proof}

\begin{proof}[Proof of Lemma~\ref{lem:neighbor-size}]
We first address $\cB$. Note that for a given $i\in[\lfloor t\rfloor]$, we have
\[\mb{P}\bigg(\sum_{j=1}^{\lfloor t\rfloor}\delta_{i,j}>2d+x\bigg)\le\exp(-c(d+x))\]
by Chernoff. By Chernoff again, for every $k$ we see
\[\#\bigg\{i\in[\lfloor t\rfloor]\colon\sum_{j=1}^{\lceil t\rceil}\delta_{i,j}\ge C\log(n/k)\bigg\}\le n(k/n)^2+(\log n)^2\]
with probability at least $1-n^{-9}$. Additionally, every such sum is bounded by $d + O(\log n)$ with probability at least $1-n^{-9}$. We can now see $\cB$ holds with sufficiently good probability by considering for each possible size of $S$ the sum of the largest rows (and similarly for the columns).

For $\cQ$, note that $\mb{E}|\xi_{i,j}|^2\le1+\beta^{-2}\le 2\beta^{-2}$, so we have $\mb{P}(|\xi_{i,j}|>8H/\beta)\le 1/(32H^2)$, and hence $\mb{P}(\delta_{i,j}|\xi_{i,j}|>8H/\beta)\le p/(32H^2)$. By the Chernoff bound, there are at most $2pn^2/(32H^2)+(\log n)^2$ entries of $A_t$ which are greater than $8H/\beta$ in magnitude with probability at least $1-\exp(-(\log n)^2)$. Now we take a union bound over a dyadically separated set of $H\in[1,n^4]$ and appropriately adjust constants to obtain $\mb{P}(\cQ)\le 1-1/(3n^3)$.

For $\mc{R}$, note that we may assume $\ell\le n/2$ (else the bound is vacuous) and thus $L/\beta\ge d^5\beta^{-5}\ge d^2\mb{E}|\xi|$. Furthermore note that 
\[\mb{E}[\delta_{i,j}|\xi_{i,j}|] = p\mb{E}|\xi_{i,j}|\le p\beta^{-1}\text{ and }\on{Var}[\delta_{i,j}|\xi_{i,j}|]\le \mb{E}[\delta_{i,j}^2|\xi_{i,j}|^2]\le p(1+\beta^{-2})\le 2p\beta^{-2}.\]
Thus we have by Chebyshev's inequality
\[\mb{P}\bigg(\sum_{j=1}^{\lceil t\rceil}|(A_t)_{i,j}|\ge L/(2\beta)\bigg)\le \frac{2pn\beta^{-2}}{ (L/(2\beta) - pn\mb{E}[|\xi_{i,j}|])^2}\le 2pn\beta^{-2}\cdot (L/(4\beta))^{-2} = 32dL^{-2}.\]
The desired result then follows by Chernoff applied to each row and considering a union bound over values of $\ell$ along an exponential sequence, for $\ell>n^{1/9}$. (When $\ell$ is sufficiently small we can use $\cQ$ instead.)
\end{proof}

\begin{proof}[Proof of Fact~\ref{fact:number-of-iterations}]
If $k\ge (n/(2d))\log_{(2)} d $ then the result is trivial, so we assume the opposite.

Now we consider how many steps it takes us to double the value of $k_t$ in two different ranges.  When $\ell\le n\exp(-d)$ we see $d\le\log(n/\ell)$ hence
\[g(\ell)\ge (cd\ell/2)(\log(n/\ell))^{-3}.\]
The right side is increasing in $\ell$, thus it takes at most $O((\log(n/\ell))^3)$ steps for this recurrence to go from $\ell$ to a value which is at least size $2\ell$.  

When $n\exp(-d)\le\ell\le (n/(2d)) \log_{(2)} d$ we have $d\ge\log(n/\ell)$ so
\[g(\ell)\ge  (c\ell/2) (\log(n/\ell))^{-2}.\]
Thus it takes at most $O((\log(n/\ell))^2)$ steps to go from $\ell$ to at least $2\ell$. Putting these two observations together we have
\[\tau\le O\bigg(\sum_{a \geq 0} (\log  (n/(2^ak)))^3\bigg) = O((\log(n/k))^4),\]
as desired.
\end{proof}

We now prove the required anti-concentration inequality Lemma~\ref{lem:anti-concentration-bernoulli}; for the sake of simplicity we define for a (real or complex) random variable $\Gamma$,
\[\mc{L}(\Gamma,t) = \sup_{z\in \mb{C}}\mb{P}(|\Gamma-z|\le t).\]
We require the following anti-concentration inequality due to L\'evy--Kolmogorov--Rogozin \cite{Kol58,Rog61}.
\begin{lemma}\label{lem:LKR}
Let $\xi_1,\ldots,\xi_n$ be independent real-valued random variables. For any real numbers $r_1,\ldots,r_n > 0$ and any real $r \ge \max_{1\le i\le n}r_i$, we have
\begin{align*}
\mc{L}\bigg(\sum_{i=1}^{n}\xi_i, r\bigg) \le \frac{Cr}{\sqrt{\sum_{i=1}^{n}(1-\mc{L}(\xi_i, r_i))r_i^2}}
\end{align*}
for an absolute constant $C>0$.
\end{lemma}

We will also require the following basic observation regarding the distribution of certain complex random variables.
\begin{lemma}\label{lem:basic-complex}
If $\mc{L}(\xi,\beta)\le 1-\beta$ and $z\in \mb{C}\setminus\{0\}$. We have either 
\[\mc{L}(\Re(z\xi), \beta|z|/\sqrt{2})\le 1-\beta/2\quad\emph{or}\quad\mc{L}(\Im(z\xi), \beta|z|/\sqrt{2})\le 1-\beta/2.\]
\end{lemma}
\begin{proof}
Suppose not. Then there exist $x,y\in\mb{R}$ such that
\[|\Re(z\xi)-x|\le\beta|z|/\sqrt{2},\quad|\Im(z\xi)-y|\le\beta|z|/\sqrt{2}\]
with probability greater than $1-\beta$. This implies that 
\[|z\xi-x-yi|\le \beta|z|\]
with probability greater than $1-\beta$; dividing by $z$ we obtain a contradiction.
\end{proof}

We now prove Lemma~\ref{lem:anti-concentration-bernoulli}.
\begin{proof}[{Proof of Lemma~\ref{lem:anti-concentration-bernoulli}}]
Without loss of generality, let $|v_1|\ge|v_2|\ge\cdots\ge|v_n|$ so that $|v_k|\ge\rho$ and let $\rho'= \beta\rho/\sqrt{2}$. By Lemma~\ref{lem:basic-complex}, at least $k/2$ coordinates $j\in[k]$ one of 
\[\mc{L}(\Re(v_j\xi),\rho')\le 1-\beta/2\quad\text{or}\quad\mc{L}(\Im(v_j\xi),\rho')\le 1-\beta/2.\]
By multiplying $\xi$ by $\sqrt{-1}$ if necessary, we may assume the first case holds, and let the set of such coordinates be $S\subseteq[k]$. This immediately implies that 
\[\mc{L}(\Re(v_j\delta_j\xi_j),\rho')\le 1-\beta p/2\]
for all $j\in S$. By Lemma~\ref{lem:LKR}, for appropriately chosen $C$ we have  
\begin{align*}
\mc{L}\bigg(\sum_{j=1}^nv_j\delta_j\xi_j,r\bigg)&\le\mc{L}\bigg(\sum_{j\in S}v_j\delta_j\xi_j,r\bigg)\le\mc{L}\bigg(\sum_{j\in S}\Re(v_j\delta_j\xi_j),r\bigg)\\
&\le\frac{Cr}{\big(\rho^{\prime2}\sum_{i\in S}(1-\mc{L}(\Re(v_i\delta_i\xi_i),\rho')\big)^{1/2}}\le\frac{Cr}{\rho'((kp\beta)/2)^{1/2}}=\frac{2C\beta^{-3/2}r}{(kp)^{1/2}\rho}
\end{align*}
where we have used $r\ge\beta\rho/\sqrt{2}=\rho'$.
\end{proof}
\vspace{-5mm}
\bibliographystyle{amsplain0}
\bibliography{main}

\end{document}